\documentclass[10pt,a4paper]{amsart}

\usepackage{amsfonts,amsthm,amssymb,amsmath,amsthm,latexsym,wasysym,stmaryrd,mathrsfs,%dsfont,
txfonts,amscd}
\usepackage{rotating,fancyhdr,color,enumerate,graphicx,hyperref,a4wide,multicol}
\usepackage[centerlast]{caption2}
\usepackage[english]{babel}
\usepackage[utf8]{inputenc}
\usepackage{multirow}
\usepackage{tabularx}

\usepackage[T1]{fontenc}

\graphicspath{{Figures/}}

\allowdisplaybreaks
\usepackage{url}
\makeatletter
\def\url@leostyle{%
  \@ifundefined{selectfont}{\def\UrlFont{\sf}}{\def\UrlFont{\small\ttfamily}}}
\makeatother
\newtheorem{theorem}{Theorem}[section]
\newtheorem{lemma}[theorem]{Lemma}
\newtheorem{proposition}[theorem]{Proposition}

\newenvironment{remark}[1][Remark]{\begin{trivlist}
\item[\hskip \labelsep {\bfseries #1}]}{\end{trivlist}}

\def\e{\varepsilon}
\def\Z{{\mathbb Z}}
\def\Q{{\mathbb Q}}

\newcommand{\saut}{\vspace{\baselineskip}}

\def\fract#1/#2{\hbox{\leavevmode
  \kern.1em \raise .25ex \hbox{\the\scriptfont0 $#1$}\kern-.1em }\big/
  {\hbox{\kern-.15em \lower .5ex \hbox{\the\scriptfont0 $#2$}} }}

\def\CDZ{$\textrm{C}_2^0$}
\def\CDU{$\textrm{C}_2^1$}
\def\CDD{$\textrm{C}_2^2$}
\def\CDP{$\textrm{C}'_2$}
\def\CDS{$\textrm{C}''_2$}

\def\CUU{$\textrm{C}_1^1$}
\def\CUD{$\textrm{C}_1^2$}
\def\CUT{$\textrm{C}_1^3$}

\def\CIU{$\textrm{C}_\infty^1$}
\def\CID{$\textrm{C}_\infty^2$}
\def\CIT{$\textrm{C}_\infty^3$}

\def\tCL{$3\textrm{CL}$}

\def\Tor{T}

\newcommand{\lens}[2]{L({\scriptstyle #1},{\scriptstyle #2})}

\newcommand{\seifdue}[5]{\big(#1,({\scriptstyle #2},{\scriptstyle #3}),
                       ({\scriptstyle #4},{\scriptstyle #5})\big)}

\newcommand{\seiftre}[7]{\big(#1,({\scriptstyle #2},{\scriptstyle #3}),
                       ({\scriptstyle #4},{\scriptstyle #5}),
                       ({\scriptstyle #6},{\scriptstyle #7})\big)}

\newcommand{\bigu}[4]{\bigcup\nolimits_{{\tiny{\matr {#1} {#2} {#3} {#4}}}\phantom{\Big|}\!\!}}

\newcommand{\matr} [4] {\left(\begin{array}{@{}c@{\ }c@{}} #1 & #2 \\ #3 & #4 \\ \end{array} \right)}

\author{Benjamin Audoux}
\address{Aix Marseille Univ, CNRS, Centrale Marseille, I2M, Marseille, France}
\email{Benjamin.Audoux@univ-amu.fr}
\author{Ana G.\ Lecuona}
\address{Aix Marseille Univ, CNRS, Centrale Marseille, I2M, Marseille, France}
\email{ana.lecuona@univ-amu.fr}
\author{Fionntan Roukema}
\address{School of Mathematics and Statistics, University of Sheffield, UK}
\email{f.roukema@sheffield.ac.uk}

\begin{document}

\title{On hyperbolic knots in $S^{3}$ with exceptional surgeries at maximal distance}

\maketitle

\begin{abstract} 
Baker showed that 10 of the 12 classes of Berge knots are obtained by surgery on the minimally 
twisted 5--chain link. In this article we enumerate all hyperbolic knots in $S^3$ obtained by 
surgery on the minimally twisted 5--chain link that realise the maximal known distances between 
slopes corresponding to exceptional (lens, lens), (lens, toroidal) and (lens, Seifert fibred) 
pairs. In light of Baker's work, the classification in this paper conjecturally accounts 
for ``most" hyperbolic knots in $S^3$ realising the maximal distance between these exceptional 
pairs.
As a byproduct, we obtain that all examples that arise from the
5--chain link actually arise from the magic manifold.
%All examples obtained in our classification are realised by filling the magic manifold.
The  classification highlights additional examples not mentioned in Martelli and Petronio's survey of the exceptional fillings on the magic manifold. Of particular interest, is an example of a knot with two lens space surgeries that is not obtained by filling the Berge manifold. 
\end{abstract}

\section{Introduction}\label{intro} 

Thurston's ground-breaking work in the 1970s showed that every non-trivial knot that is not a satellite is hyperbolic, and that 
non-hyperbolic surgeries on such knots are ``exceptional". These deep and surprising results reinvented the field of hyperbolic geometry and knot theory. 
With the exception of $S^3$, given a non-hyperbolic manifold $M$ the set of all cusped hyperbolic manifolds with $M$ as 
a filling is unwieldy, and we shouldn't expect to be able to write down the set of all hyperbolic manifolds which have 
a lens space filling. However, in light of Thurston's work, it becomes reasonable to ask which hyperbolic knots in $S^3$ 
have a lens space surgery, or which hyperbolic knots have a toroidal filling that is ``far" from the $S^3$ filling. 
This paper looks at hyperbolic knots in $S^3$ that have exceptional fillings that are ``far" apart.

Let $K$ be a knot in $S^{3}$ and consider its exterior $S^{3}\setminus\nu(K)$ where $\nu(K)$ is a small open regular neighborhood of the knot. For a slope $\alpha$ (the isotopy class of an essential simple closed curve) on the boundary of the exterior of $K$, the closed manifold obtained from $\alpha$-surgery (gluing a solid torus to the exterior of $K$ by identifying the meridian to $\alpha$) is denoted by $K(\alpha)$.

Suppose that $K$ is hyperbolic, that is, its complement admits a Riemannian metric of constant sectional curvature $-1$ which is complete and of finite volume. Then Thurston's hyperbolic Dehn surgery theorem implies that all but finitely many slopes produce hyperbolic manifolds via surgery see \cite{b:th} and \cite{Pet}. The exceptional cases are called exceptional slopes and exceptional surgeries. 

It is a consequence of the geometrization theorem that every exceptional 
surgery on a hyperbolic \emph{link} is either $S^3$, a lens space, has an essential surface of non-negative Euler 
characteristic, or fibres over the sphere with three exceptional fibres. We now assign the 
following standard names to these classes of non-hyperbolic
3-manifolds following \cite{Gor'}. We say that a manifold is of type
$D$, $A$, $S$ or $T$ if it contains, respectively, an essential disc,
annulus, sphere or torus, and of type $S^{H}$ or $T^{H}$ if it
contains a Heegaard sphere or torus. Finally we denote by $Z$ the type
of small closed Seifert manifolds. Notice that $S^{H}=\big\{S^{3}\big\}$
and that $T^{H}$ is the set of lens spaces (including $S^{1}\times S^{2}$). 

In the present paper, we are interested in hyperbolic manifolds $X$ with a
torus boundary component $\tau$ supporting a pair $(\alpha,\beta)$ of exceptional slopes whose
associated surgeries lead respectively to manifolds of types $\mathcal{C}_1$ and $\mathcal{C}_2$, where $\mathcal{C}_1, \mathcal{C}_2\in\big\{S^H, S, T^H, T, D, A, Z\big\}$, the set of exceptional type manifolds described above. We will summarize this situation by writing $(X,\tau;\alpha,\beta)\in(\mathcal{C}_1,\mathcal{C}_2)$. The distance (minimal geometric intersection) between two slopes
$\alpha$ and $\beta$ on a torus is denoted by $\Delta(\alpha,\beta)$. The
\emph{maximal distance between types of exceptional manifolds}
$\mathcal{C}_1$ and $\mathcal{C}_2$ is defined as the
$\max\big\{\Delta(\alpha,\beta) \, |\, (X, \tau;\alpha,\beta)\in
(\mathcal{C}_1,\mathcal{C}_2)\big\}$ and denoted by $\Delta(\mathcal{C}_1,\mathcal{C}_2)$.

Quite some energy has been devoted in the literature to the understanding of exceptional slopes on hyperbolic manifolds. In the case  of hyperbolic knot 
exteriors there are strong restrictions on their exceptional surgeries or  fillings. The $S^H$--filling is unique \cite{GL} and no knot exterior 
has a filling with an essential annulus or disc. Conjecturally no 
hyperbolic knot exterior has a reducible surgery \cite{Cabling}. So, there are nine 
possible exceptional pairs obtained by surgery on a hyperbolic knot in
$S^3$, namely the $(S^H,T^H)$, $(S^H,T)$, $(S^H,Z)$, 
$(T^H,T^H)$, $(T^H,T)$, $(T^H,Z)$, $(T,T)$, $(T,Z)$ and $(Z,Z)$ exceptional pairs. 

The $(S^H,T)$ pairs have been completely enumerated \cite{tor_class}. Examples of $(S^H,Z)$ pairs have been constructed, see for example \cite{EM,excep_slopes}. 
The exceptional surgeries on the figure eight knot tell us that 
$\Delta(T^H,Z),\Delta(T,Z),\Delta(Z,Z)>5$, and from \cite{Agol} we know that there are only a finite number of examples realising these distances. The $(S^H,T^H)$ pairs are conjecturally a subset of the Berge knots classified in \cite{Berge}. It follows that, since  the remaining three cases all involve a $T^H$ surgery, an enumeration of the remaining 
three exceptional pairs is conjecturally an enumeration of a subset of Berge knots. Baker showed \cite{Baker} that 10 of the 12 
classes of Berge knots are obtained by surgery on the minimally twisted 5--chain link (5CL, see Figure~\ref{chain_links_fig}). So, conjecturally, most of the 
hyperbolic knots realising $(T^H,T^H)$, $(T^H,T)$ and $(T^H,Z)$ exceptional pairs of slopes are 
obtained by surgery on 5CL. 

In this article we enumerate all hyperbolic knots obtained from
surgery on the 5CL that realise the maximum known distance between the
exceptional filling types. We completely classify the knots arising in
this manner and having either two different lens space surgeries; a
lens space surgery and a toroidal surgery at distance 3; or a lens space surgery and a small Seifert surgery at distance 2. In light of Baker's work, the classification in this article 
conjecturally accounts for most examples of hyperbolic knots with an exceptional pair of slopes at maximal distance.  Our main result is the following:

\begin{figure}
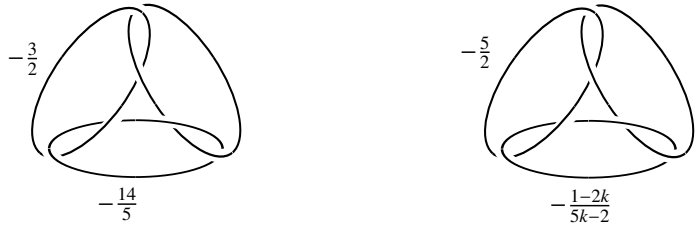
\caption{Surgery presentation for all (distinct) hyperbolic knots  with two lens space fillings obtained by surgery 
on the minimally twisted 5--chain link.}
\label{lens_lens_fig}
\end{figure}

\begin{figure}
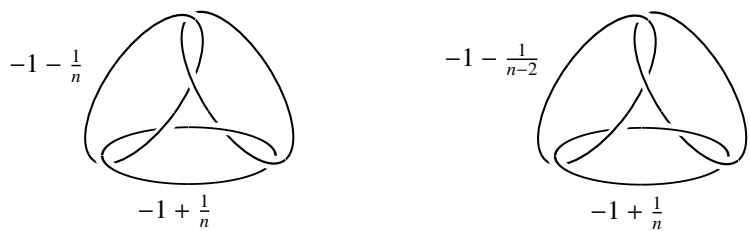
\caption{Surgery presentation for all (distinct) hyperbolic knots with
  a lens space and a toroidal filling at distance 3, or a lens space
  and a Seifert filling at distance 2, obtained by surgery on 5CL.}
\label{lens_toroidal_Z_fig}
\end{figure}

\begin{theorem}\label{thm1}
Let $K$ be a hyperbolic knot in $S^3$ obtained by surgery on the minimally twisted 5--chain 
link with two exceptional slopes $\alpha$ and $\beta$, and such that $K(\alpha)$ is a 
lens space.
\begin{itemize} 
\item If $K(\beta)$ is a lens space, then $K$ is found in Figure \ref{lens_lens_fig}.
\item If $K(\beta)$ is toroidal, then the distance between $\alpha$ and $\beta$ is at most three and 
if the distance equals three, then $K$ is found in Figure \ref{lens_toroidal_Z_fig}.
\item If $K(\beta)$ fibres over the sphere with three exceptional 
fibres then the distance between $\alpha$ and $\beta$ is at most two and 
if the distance equals two then $K$ is found in Figure \ref{lens_toroidal_Z_fig}.
\end{itemize}
\end{theorem} 

Given $M$, an orientable cusped hyperbolic $3$--manifold and $\tau$ a fixed torus component of the boundary of its compactification, it is a consequence of \cite{Lack:setsize}  that 8 is a universal upper bound for $\Delta(\alpha_1, \alpha_2)$ for each exceptional pair $(M, \tau; \alpha_1, \alpha_2)$. The celebrated Gordon-Luecke theorem \cite{GL} can be formulated by saying that 
$\Delta(S^H, S^H)=0$, the Cabling conjecture by saying that $\Delta(S, S^H)=-\infty$ 
\cite{Cabling}, the Berge conjecture implies that the Berge knots in \cite{Berge} 
contain all exceptional pairs of type $(S^H, T^H)$, and the theorem of \cite{tor_class} 
by saying that the knots realizing $\Delta(\alpha_1, \alpha_2)=\Delta(S^H, T)$ are 
precisely the Eudave-Mu\~noz knots. 

It is natural to generalise these types of questions by asking whether we can find 
$\Delta(\mathcal{C}_1, \mathcal{C}_2)$ for each pair of classes 
$\mathcal{C}_1, \mathcal{C}_2\in\big\{S^H, S, T^H, T, D, A, Z\big\}$, and whether we can enumerate 
all $(M, \tau;\alpha_1, \alpha_2)$ of type $(\mathcal{C}_1, \mathcal{C}_2)$ with 
$\Delta(\alpha_1, \alpha_2)=\Delta(\mathcal{C}_1, \mathcal{C}_2)$. A great deal 
is known, see \cite{essential:tori} or \cite{Gor} for an overview. 

If a knot in $S^3$ is not a torus knot or a satellite knot then its exterior is a hyperbolic 3-manifold. 
We can consider all $(M_K, \tau;\alpha_1, \alpha_2)$ when $M_K$ is the exterior of a knot $K$ 
in $S^3$ and ask what is $\Delta(\mathcal{C}_1, \mathcal{C}_2)$ and which $(M_K, \tau;\alpha_1, \alpha_2)$ 
of type $(\mathcal{C}_1, \mathcal{C}_2)$ have $\Delta(\alpha_1, \alpha_2)=\Delta(\mathcal{C}_1, \mathcal{C}_2)$ 
for this subclass of hyperbolic manifolds. Of course, this is the same as asking what is 
the greatest value of $\Delta(\alpha_2, \alpha_3)$ among exceptional triples 
$(M, \tau; \alpha_1, \alpha_2, \alpha_3)$ of type $(S^H, \mathcal{C}_1, \mathcal{C}_2)$ and which 
$(M, \tau; \alpha_1, \alpha_2, \alpha_3)$ realise the maximum $\Delta(\alpha_2, \alpha_3)$. 
From this perspective, we enumerate in this article such $(S^H, \mathcal{C}_1, \mathcal{C}_2)$ triples obtained from the 
minimally twisted 5--chain link.

%%%%%%%%%%%%%%%%%%%%%%%%%%%%%%

In order to state some of the noteworthy remarks coming from the analysis done to establish Theorem~\ref{thm1} we need to introduce some more notation. The chain links that are ubiquitous throughout this paper are depicted in Figure \ref{chain_links_fig}. We keep the notation 
conventions of \cite{MPR} and denote the minimally twisted 5--chain link by 5CL and 
its exterior by $M_5$; the 4--chain link is denoted by 4CL and its
exterior is denoted by $M_4$. 
The minimally twisted 4--chain link $\text{M4CL}$, as
well as its exterior $F$, will also appear extensively in the text.
A $(-1)$--surgery on any component of 4CL gives a 3--chain link 3CL, whose exterior is denoted by $M_{3}$. We closely reference the tables from \cite{Magic} which give a classification of the exceptional surgeries on 
the mirror $\text{3CL}^{\ast}$ shown in Figure \ref{chain_links_fig}. The exterior of this link is 
the ``magic manifold'' \cite{GW} which we will denote by $N$.

\begin{figure}[h!]
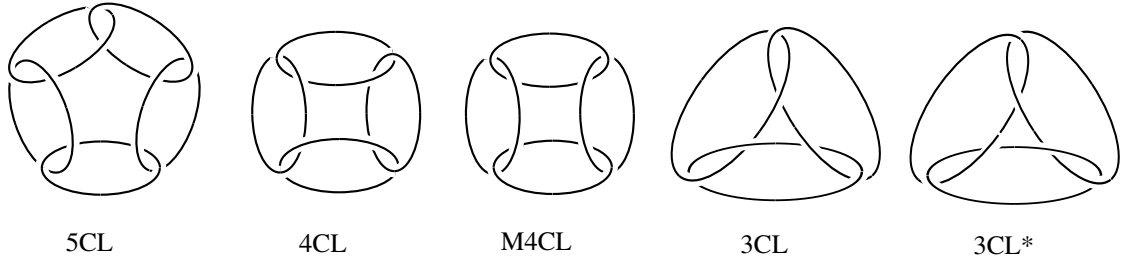
\caption{The minimally twisted 5--chain link 5CL, the 4--chain link 4CL, the minimally twisted 4--chain link M4CL and the 3--chain links 3CL and $\text{3CL}^{\ast}$. The exteriors of these links are respectively called $M_{5},M_{4},F,M_{3}$ and $N$.}
\label{chain_links_fig}
\end{figure}

The knots in Figures~\ref{lens_lens_fig} and~\ref{lens_toroidal_Z_fig} are described by giving a filling instruction on two of the 3 boundary components of the magic manifold. The exceptional slopes on the knots $N(-\frac32,-\frac{14}5)$ and 
$N(-\frac52,\frac{1-2k}{5k-2})$ from Figure \ref{lens_lens_fig} and
the corresponding fillings 
are found in Theorem~\ref{lens_lens_thm}. The exceptional slopes on 
$N(-1+\frac1n,-1-\frac1n)$ and $N(-1+\frac1n,-1-\frac1{n-2})$ from Figure \ref{lens_toroidal_Z_fig} 
and the corresponding fillings are found in Theorem \ref{lens_toroidal_prop}. 
Theorems~\ref{lens_lens_thm} 
and \ref{lens_toroidal_prop} go further and show that the three families of knots and the isolated 
example shown in Figures \ref{lens_lens_fig} and \ref{lens_toroidal_Z_fig} are all distinct knots. 

There is a unique hyperbolic knot in a torus with two non-trivial surgeries \cite{Berge_torus}; 
the exterior of this knot is called the Berge manifold, which will appear frequently in the text. It can be obtained by filling one of the 3 boundary components of the magic manifold $N$. Indeed, the Berge manifold is $N(-\frac{5}{2})$. Cutting, twisting and filling 
the boundary of the torus yields an infinite family of inequivalent knots in $S^3$ with two 
lens space fillings. This family is precisely the set of $N(-\frac52,\frac{1-2k}{5k-2})$ 
from Theorem \ref{thm1}.
%We find particularly interesting the following fact.
It should be highlighted that the example $N(-\frac32,-\frac{14}5)$ is not obtained by surgery 
on the Berge manifold (Theorem \ref{lens_lens_thm}).

The article \cite{3lens} contains a complete description of all surgeries on the 5CL with 
three cyclic fillings, which are fillings leading to type $S^H$ or
$T^H$ manifolds. It is a more general question than our quest to find \emph{knot exteriors} on the 5CL with two cyclic fillings and the techniques used in \cite{3lens} to reduce the argument to an analysis of the fillings on the Magic manifold are different from ours. A translation of our results into the language of \cite{3lens} follows. The family $\{N(-\frac52,\frac{1-2k}{5k-2})\}$ is 
the family $\{B_{(2k-1)/(5k-2)}\}\subset \{B_{p/q}\}$ from \cite{3lens}, and 
the isolated example $N(-\frac32,-\frac{14}5)$ is $A_{2,3}$ from \cite{3lens}.
As a final remark, let us emphasize the fact that the family $N(-1+\frac1n,-1-\frac1n)$ and its exceptional 
slopes and fillings are highlighted in \cite[Table~17]{Magic}, but the
distinct family $N(-1+\frac1n,-1-\frac1{n-2})$ is not.

\subsection{Article structure} The results in this article are obtained 
by a careful analysis of the classifications of exceptional sets of slopes 
on surgeries of the minimally twisted 5--chain link given in \cite{Magic} and 
\cite{excep_slopes}. The work done there, translates the enumeration of exceptional 
pairs realising maximal distances into finding the solutions to a (long) list 
of elementary diophantine equations. The translation necessitates a table by table analysis of the work given in 
\cite{Magic} and \cite{excep_slopes}. A collection of easy (but technical) 
lemmata in the Appendix facilitates the translation and reduces the amount of 
work needed. The proofs of the main results are littered with references to 
results in the Appendix, \cite{Magic}, and \cite{excep_slopes}. Most
equalities and isomorphisms shall, for instance, be subscripted by a
label that refers to the Appendix. Therefore, 
this article is best read with both articles and the Appendix in-hand. 

Section \ref{notation_sec} sets out the notation and conventions used 
throughout this article. Section \ref{lens_lens_sec} gives an enumeration of all 
exceptional $(S^H,T^H,T^H)$ triples obtained by surgery on 5CL. 
Section \ref{lens_toroidal_sec} gives an enumeration of all 
exceptional $(S^H,T^H,T)$ triples obtained by surgery on 5CL. 
Section \ref{lens_Z_sec} gives an enumeration of all 
exceptional $(S^H,T^H,Z)$ triples obtained by surgery on 5CL. 
Sections \ref{lens_lens_sec}-\ref{lens_Z_sec} all proceed in the same way. The 
sections start with a precise statement about the enumeration of the exceptional triples.  
The results are established by first showing that all examples are obtained by 
surgery on 4CL, and then showing that all examples are obtained by surgery on 3CL.  
The final sections then enumerate all examples of exceptional triples obtained by 
surgery on 3CL.  

\subsection{Acknowledgements and remarks} We are thankful to Daniel Matignon and Luisa Paoluzzi for fruitful conversations. We also want to thank Ken Baker and Marc Lackenby for interesting discussions. The second author is partially supported by the Spanish GEOR-MTM2014-55565. The third author was supported as a member of the Italian FIRB project 
`Geometry and topology of low-dimensional manifolds' (RBFR10GHHH), internal funding 
from the University of Sheffield and the French ANR research project ``VasKho'' ANR-11-JS01-002-01.

\section{Notation and conventions}\label{notation_sec}

In this section we set out notation and conventions used throughout the article. 
We will use the conventions on surgery instructions set out in \cite{excep_slopes} which 
we briefly outline. For more detailed descriptions, please refer to
\cite{excep_slopes}. Given an orientable compact 3-manifold $X$ such
that $\partial X$ is a collection of tori, we use the term
\emph{slope} to indicate the isotopy class of a non-trivial unoriented
essential simple closed curve on a component of $\partial X$. After fixing a choice of meridian
and longitude on a boundary torus, a slope is naturally identified
with an element in $\Q\cup\{\infty\}$. A \emph{filling instruction}
$\alpha$ (also denoted by $\mathcal F$) for $X$ is a set consisting of either a slope or the empty
set for each component of $\partial X$. The chain links have a rotational symmetry which allow us to unambiguously choose any component as the first component and order the remaining cyclically in the anticlockwise direction. The filling instructions shall then be identified with tuples of elements in $\mathbb{Q}\cup\{\infty\}$. The \emph{filling} $X(\alpha)$ is the manifold obtained by gluing one solid torus to $\partial X$ for each non-empty slope in $\alpha$. The meridian of the solid torus is glued to the slope.

A very related concept to that of a filling is a \emph{surgery} on a link $L\subset S^{3}$. By definition, a surgery on $L$ is a filling of the exterior of $L$, $S^{3}\setminus\nu(L)$, where $\nu(L)$ is an open regular neighborhood of $L$. By a \emph{surgery instruction} for $L$ we mean a filling instruction on the exterior of $L$.

In the present article we will be concerned with \emph{exceptional
  fillings}. If the interior of $X$ is hyperbolic but the interior of
$X(\alpha)$ is not, we say that $\alpha$ is an \emph{exceptional
  filling instruction} for $X$ and $X(\alpha)$ is an exceptional
filling. If the resulting manifold has a cyclic fundamental group,
that is if it belongs to $S^H$ or $T^H$, then we say that it is furthermore
\emph{cyclic}. We follow the notation used to describe the sets of exceptional slopes set out in 
\cite{Gor}. The set of exceptional slopes on a fixed toroidal boundary component 
$\tau$ of a hyperbolic 3-manifold $X$ is denoted by $E_{\tau}(X)$, and the cardinality 
of $E_{\tau}(X)$ by $e_{\tau}(X)$. In our case $\tau$ will refer to the $n^{th}$ component 
of the chain link with $n$ components and is dropped throughout the article. A word of caution: 
when $\mathcal F$ is a filling instruction on $M_5$, we write the elements of $E(M_5(\mathcal F))$ with 
respect to the choice of bases on $M_5$ (\emph{and not} $M_5(\mathcal F)$!). 

Beyond the exceptional pairs $(\mathcal{C}_1,\mathcal{C}_2)$ explained in the introduction, we will work also with \emph{exceptional $(\mathcal{C}_1,...,\mathcal{C}_n)$ 
$n$-tuples}. By this we mean the following: if $X$ is a hyperbolic 3-manifold and $\alpha_1,\dots,\alpha_n$ are exceptional slopes on a fixed toroidal boundary component of $X$, with $X(\alpha_i)$ a manifold 
of type $\mathcal{C}_i$, then we say that 
$(X,\alpha_1,...,\alpha_n)$ is an exceptional $(\mathcal{C}_1,...,\mathcal{C}_n)$ 
$n$-tuple and write $(X,\alpha_1,...,\alpha_n)\in(\mathcal{C}_1,...,\mathcal{C}_n)$. There is a notion of equivalence among exceptional tuples. We will say that two exceptional $n$-tuples $(X_1,\alpha_1,...,\alpha_n)$ and $(X_2,\beta_1,...,\beta_n)$ 
are \emph{equivalent} if there exists a homeomorphism $h:X_1\rightarrow X_2$ with 
$X_2(h(\alpha_i))=X_2(\beta_i)$. When two $n$-tuples $(X_1,\alpha_1,...,\alpha_n)$, 
$(X_2,\beta_1,...,\beta_n)$ are equivalent we write 
$(X_1,\alpha_1,...,\alpha_n)\cong(X_2,\beta_1,...,\beta_n)$.

We now recall the following important notion introduced 
in \cite{excep_slopes}: given $\alpha$, a filling instruction on a manifold $X$, we say that $\alpha$ 
\emph{factors through a manifold $Y$} if there exists some filling instruction $\alpha' \subset \alpha$ such that $Y=X(\alpha')$. 

To describe the exceptional fillings on the minimally twisted 5--chain link, we follow the standard choice of notation used to describe graph manifolds set out in \cite{excep_slopes}. Very briefly, if $G$ is an orientable surface with $k$ boundary components and $\Sigma$ is $G$ minus $n$ discs, we can construct homology bases $\{(\mu_{i},\lambda_{i})\}$ on $\partial(\Sigma\times S^{1})$. For coprime pairs $\{(p_{i},q_{i})\}_{i=1}^{n}$ with $|p_{i}|\geq 2$ we get a Seifert manifold $(G,(p_{1},q_{1}),\dots,(p_{n},q_{n}))$ with fixed homology bases on its $k$ boundary components. Given Seifert manifolds $X$ and $Y$ with boundary and orientable base surfaces as above and an element $B\in\mathrm{GL}_{2}(\Z)$, we define $X\cup_{B}Y$ unambiguously to be the quotient manifold $X\cup_{f}Y$, with $f:T\rightarrow T'$ where $T$ and $T'$ are arbitrary boundary tori of $X$ and $Y$, and $f$ acting on homology by $B$ with respect to the fixed bases. Similarly one can define $X\big/_{B}$ when $X$ has at least two boundary components. 

As is common in the litterature, we employ a somehow more
flexible notation for lens spaces than the usual one. We will write
$\mathrm{L} (2,q)$ for the real projective space, $\mathrm{L} (1,q)$
for the 3-sphere, $\mathrm{L} (0,q)$ for $S^{2}\times S^{1}$ and
$\mathrm{L} (p,q)$ for $\mathrm{L}(|p|,q')$ with $q\equiv q'$ modulo
$p$ and $0<q'<|p|$, for any coprime $p,q$. Later in the paper, we will
often be interested in understanding when $\mathrm{L}(x,y)=S^3$ where
$x$ and $y$ shall have some complicated expression; as $\mathrm{L}(x,y)=S^3$ if and only if $|x|=1$, we will often replace $``y"$ with $``\star"$ to simplify matters. 

Finally, throughout the text the symbols $\e$, $\e_1$, $\eta$,
etc. will all denote $\pm1$, and $k$,$n$, etc. will denote integers.

\section{$(S^H,T^H,T^{H})$ triples from 5CL}\label{lens_lens_sec}

In this section, we enumerate all exceptional $(S^H,T^H,T^{H})$ triples obtained by surgery 
on the 5CL. Each one of these triples can be thought of as a knot in
$S^{3}$ with two different lens space surgeries.
%To state the main
%result of this section we need to first define the following
%$\Z$--indexed family  of triplets of exceptional slopes: $A_k:=\left( N\big(-\tfrac{5}{2},\tfrac{1-2k}{5k-2}\big),\infty,
 % -2,-1\right)$.

\begin{theorem}\label{lens_lens_thm}
\begin{itemize}
\item[]
\item If $\left(M_5(\frac ab,\frac cd,\tfrac ef,\tfrac gh),\alpha,\beta,\gamma\right)\in(S^H,T^H,T^H)$ 
then it is equivalent 
to either $\left(N(-\frac{3}{2},-\frac{14}{5}), -2,-1,\infty\right)$ or $A_n:=\left( N\big(-\tfrac{5}{2},\tfrac{1-2n}{5n-2}\big),\infty,-2,-1\right)$ for some $n\in\Z$.
\item The sets of exceptional slopes and corresponding fillings of
  $N\big(-\frac{3}{2},-\frac{14}{5})$ and of $N\big(-\frac{5}{2},\frac{1-2n}{5n-2}\big)$,
for $n\neq0$, are given in Table \ref{not_berge_example}.
\item All $ N\big(-\frac{5}{2},\frac{1-2n}{5n-2}\big)$ are obtained by
  filling the Berge manifold; but none of the $A_n$ is equivalent to $\left(N(-\frac{3}{2},-\frac{14}{5}), -2,-1,\infty\right)$. 
\end{itemize}
\end{theorem}

\begin{table}
\begin{center}
\begin{tabular}{ |c|c| }
\hline
\multicolumn{2}{ |c| }{$n\in\mathbb{Z}\backslash\{0\}$, 
\qquad $E(N(-\frac52,\frac{1-2n}{5n-2}))=\{-3,-2,-\frac32,-1,0,\infty\}$} \\
\hline\hline
$\beta\in E(N(-\frac52,\frac{1-2n}{5n-2}))$ & $N(-\frac52,\frac{1-2n}{5n-2})(\beta)$ \\ \hline\hline
$\beta=\infty$ & $S^3$ \\ \hline
$\beta=-3$ & $\seifdue D213{-2}\bigu0110 \seifdue D21{3n-1}{5n-2}$ \\ \hline
$\beta=-2$ & $\lens{18-49n}{7-19n}$ \\ \hline
$\beta=-\frac32$ & $\seifdue D2131\bigu110{-1} \seifdue D21{8n-3}{5n-2}$\\ \hline
$\beta=-1$ & $\lens{49n-19}{31n-12}$ \\ \hline
$\beta=0$ & $\seifdue D2{-1}{5n-2}{8n-3}\bigu01{-1}{-1} \seifdue D2131$\\ \hline
\hline
\multicolumn{2}{ |c| }{$E(N(-\frac32,-\frac{14}{5}))=\{-3,-\frac52,-2,-1, 0, \infty\}$} \\
\hline\hline
$\beta\in E(N(-\frac32,-\frac{14}{5}))$ & $N(-\frac32,-\frac{14}{5})(\beta)$ \\ \hline\hline
$\beta=\infty$ & $\lens{32}{-9}$ \\ \hline
$\beta=-3$ & $\seiftre{S^2}21329{-5}$\\ \hline
$\beta=-\frac52$ & $\seifdue D2131\bigu110{-1} \seifdue D214{-5}$\\ \hline
$\beta=-2$ & $S^3$\\ \hline
$\beta=-1$ & $\lens{31}{17}$ \\ \hline
$\beta=0$ & $\seifdue D215{-4}\bigu01{-1}{-1} \seifdue D2131$\\ \hline
\end{tabular}
\end{center}
\caption{The exceptional slopes and corresponding fillings of hyperbolic knot exteriors 
in $S^3$ with two lens space fillings obtained by surgery on 5CL. \label{not_berge_example}}
\end{table}

\begin{remark}
If $n=0$ then $N(-\frac52,\frac{1-2n}{5n-2})$ is the exterior of the $(-2,3,7)$ 
pretzel knot which has 7 exceptional slopes, see \cite[Table~A.4]{Magic} for details. 
\end{remark}

We prove Theorem \ref{lens_lens_thm} by first considering, in Section \ref{lens_lens5CL_subsec}, all $\left(M_{5}(\mathcal F),\alpha,\beta,\gamma\right)\in(S^H,T^H,T^{H})$ with $\mathcal{F}$ not factoring through $M_4$. This set will turn out to be empty and we proceed in Section \ref{lens_lens4CL_subsec} to investigate the $\left(M_{4}(\mathcal F),\alpha,\beta,\gamma\right)\in(S^H,T^H,T^{H})$ with $\mathcal{F}$ not factoring through $M_3$. Again, there will be no such examples and we will finally consider in Section~\ref{lens_lens3CL_subsec} the case $\left(M_{3}(\mathcal F),\alpha,\beta,\gamma\right)\in(S^H,T^H,T^{H})$ . We will produce a complete list of examples, the family $A_{n}$ and the isolated example in the statement of Theorem~\ref{lens_lens_thm}. The fact that the examples we find are all different is an easy consequence of the results in \cite{Magic} and is shown at the end of Section \ref{lens_lens3CL_subsec}. Throughout the argument, easy (but technical) lemmata from the Appendix are referenced. 

\subsection{Hyperbolic knots with two lens surgeries arising from the 5--chain link}\label{lens_lens5CL_subsec}

In this section we prove that if $M_5(\frac ab,\frac cd,\frac ef,\frac
gh)$ is  a hyperbolic knot exterior admitting two different lens space fillings then the instruction $(\frac ab,\frac cd,\frac ef,\frac gh)$ 
factors through $M_4$.
% \begin{figure}[h!]
% \input{5CL.pdf_tex}
% \end{figure}
If $\left(M_5(\frac ab,\frac cd,\frac ef,\frac gh),\alpha, \beta, \gamma\right)\in (S^H,T^H,T^{H})$ and 
$(\frac ab,\frac cd,\frac ef,\frac gh)$ does not factor through $M_4$ then \cite[Theorem~4]{excep_slopes} 
tells that there are two different scenarios to consider: either
there are 3 exceptional slopes, or there are more.
We study separately these two cases, starting with the latter one.  

\subsubsection{Case $e\left(M_5(\mathcal F)\right)>3$}
By application of \cite[Theorem~4]{excep_slopes}, we know that the
manifold $M_5(\frac ab,\frac cd,\frac ef,\frac gh)$ is then equivalent
to some $M_5(\mathcal F)$ listed in \cite[Tables 6\,--
11]{excep_slopes}.
A careful inspection of the corresponding exceptional sets, given in \cite[Tables 14\,--
20]{excep_slopes},
shows that Tables 17 and 18 are the only ones where types $S^H$ and
$T^H$ appear simultaneously. The only possible cases are hence
$\mathcal{F}=(-2,\frac pq,3,\frac uv)$
\cite[Tables~17]{excep_slopes} and 
$\mathcal{F}=(-2,\frac pq,\frac rs,-2)$
\cite[Tables~18]{excep_slopes}. The exceptional slopes are then
$\{-1,0,1,\infty\}$, but since $\Delta(S^H,T^H)=\Delta(T^H,T^H)=1$ \cite{CGLS},
$1$ and $-1$ cannot yield simultaneously cyclic fillings, so $0$
and $\infty$ are necessarily part of the exceptional triple.
\begin{description}
\item[If $\mathcal{F}=(-2,\frac pq,\frac rs,-2)$] $0$ is the only
  possibility for the $S^H$--filling, and then either $\frac pq$ or $\frac
  rs$ is of the form $1+\frac 1n$; but $\infty$ is also a $T^H$--filling, so
  $|s|=|q|=1$, that is $\frac pq,\frac rs\in\Z$. It follows that
  $1+\frac 1n$ is an integer, and the only possibilities are 0 or 2. But according
  Lemma \ref{5CL_factor}, if it is 0, then $M_5(\mathcal F)$ is
  non-hyperbolic, and if it is 2, then it factors through $M_4$.
\item[If $\mathcal{F}= (-2,\frac pq,3,\frac uv)$] On the one hand, $0$
  is a cyclic slope, it follows that either  $\frac uv=3$, $\frac uv=3+\frac 1k$,
  $\frac uv=\frac{6n+7}{2n+3}=3-\frac2{2n+3}$ or $|(3+2n)u-(7+6n)v|=1$
  that is $\frac uv=3-\frac2{2n+3}+\frac{\e}{(2n+3)v}$. Moreover, in
  the last two cases, $\frac pq=1+\frac 1n$ so we can assume that $n\neq
  -1,-2$ otherwise $M_5(\mathcal F)$ would be non-hyperbolic or $\mathcal F$ would factor
  through $M_4$ because of Lemma \ref{5CL_factor}. We obtain hence the lower   bound $\frac uv\geq 3-\frac 2{|2n+3|}-\frac{1}{|2n+3|}\geq 3-\frac
  33=2$.
  On the other hand, $\infty$ is also a cyclic slope, so either
  $\frac uv=\frac 13$, $\frac
  uv=\frac{2k+1}{6k+1}=\frac13+\frac{2}{3(6k+1)}$, $v-3u=\e$ that
  is $\frac uv=\frac13-\frac{\e}{3v}$, or $|(1+2k)v-(1+6k)u|=1$ that
  is $\frac uv=\frac13+\frac{2}{3(6k+1)}+\frac{\e}{(6k+1)v}$. It follows that we have the upper bound $\frac uv\leq\frac13+\frac23+1=2$.
  In conclusion, $\frac uv=2$ and $M_5(\mathcal
  F)$ factors through $M_4$ because of Lemma \ref{5CL_factor}.
\end{description}

\subsubsection{Case $e\left(M_5(\mathcal F)\right)=3$}
By application of \cite[Theorem~4]{excep_slopes}, we know that the
exceptional set of slopes is $\big\{0,1,\infty\big\}$. Moreover, we have
\begin{equation}\label{inf-F}
M_5\big(\tfrac ab, \tfrac cd, \tfrac ef, \tfrac gh\big)(\infty)\mathop{=}\limits_{(\ref{5CLinf})} 
F\big(-\tfrac ab, \tfrac fe, \tfrac dc, -\tfrac gh\big).
\end{equation}
Recall that if one of $a,b,c,d,e,f,g,h=0$ then $M_5(\frac ab,\frac cd,\frac ef,\frac gh)$ 
is non-hyperbolic by Lemma \ref{5CL_factor}. 
The enumeration of closed fillings of $F$ is found in \cite[Table 4]{excep_slopes}. We will 
use (\ref{inf-F}) to translate instructions on $M_5$ to instructions on $F$ and 
carefully consider the entries from \cite[Table 4]{excep_slopes}. In the analysis, 
T4.$n$ will denote the $n^{th}$ line of this table. 

Considering the $T^H\cup S^H$--fillings of $F$ listed in
\cite[Table~4]{excep_slopes}, and in view of Lemma \ref{D4-F}, we learn that, up to a $D_{4}$ permutation of slopes, T4.2--T4.5 is a complete list of necessary and sufficient 
conditions for $M_5(\tfrac ab, \tfrac cd, \tfrac ef, \tfrac gh)(\infty)\in S^H\cup T^H$. The lines T4.2 and T4.3, which correspond to $\tfrac pq=0$, can be ignored since by \eqref{inf-F} and Lemma~\ref{5CL_factor} they yield a non hyperbolic filling.

The entry T4.4 tells us that if 
$M_5(\tfrac ab, \tfrac cd, \tfrac ef, \tfrac gh)(\infty)=F(-\tfrac ab, \tfrac fe, \tfrac dc, -\tfrac gh) 
\in S^H\cup T^H$, then, taking into consideration the action of $D_{4}$ on $F$, one of the following conditions necessarily holds: 
\begin{align*}
(i)\,\, \tfrac ab=\tfrac 1n\,\,\& \,\, & \tfrac ef=k 
& (ii)\,\, \tfrac cd=n\,\,\& \,\,& \tfrac gh=\tfrac1k 
& (iii)\,\, \tfrac ab=\tfrac 1n\,\,\& \,\,& \tfrac gh=\tfrac1k 
& (iv)\,\, \tfrac cd=n\,\,\& \,\,\tfrac ef=k.    
\end{align*}
These conditions can all be identified using Lemma \ref{sym}. 
In fact, to identify for example case ($i$) with case ($iii$) it
suffices to remark that $\left(M_5(\tfrac ab,\tfrac cd,\tfrac ef,\tfrac gh),\infty\right)
\mathop{\cong}\limits_{(\ref{linksymeq1})^3\circ(\ref{symeq1})}
\left(M_5(\tfrac a{a-b},\tfrac {d-c}d,\tfrac hg,\tfrac fe),\infty\right)$.
In a similar way, case ($i$) can 
be identified with case ($ii$) using (\ref{linksymeq1}) and (\ref{linksymeq2}), 
and case ($i$) can be identified with case ($iv$) using (\ref{symeq8}). 

The entry T4.5 tells us that if 
$M_5(\tfrac ab, \tfrac cd, \tfrac ef, \tfrac gh)(\infty)=F(-\tfrac ab, \tfrac fe, \tfrac dc, -\tfrac gh) 
\in S^H\cup T^H$, then one of the following conditions necessarily holds: 
\begin{align*}
(i)\,\, \tfrac ab=\tfrac 1n\,\,\& \,\, &  \tfrac cd=\tfrac{\e+nk}{k}
& (ii)\,\, \tfrac ab=\tfrac k{\e+nk}\,\,\& \,\, & \tfrac cd=n 
& (iii)\,\, \tfrac gh=\tfrac 1n\,\,\& \,\, &  \tfrac ef=\tfrac{\e+nk}{k} 
& (iv)\,\, \tfrac gh=\tfrac k{\e+nk}\,\,\& \,\, & \tfrac ef=n 
\end{align*}
where $\e=\pm1$. Then Case 
($i$) is identified with Case ($iv$) using (\ref{symeq8}). Moreover, Case ($i$) is identified with Case ($iii$), and 
Case ($ii$) is identified with Case ($iv$) using (\ref{linksymeq1})$^4\circ$(\ref{linksymeq2}). 

Therefore, any $M_5(\frac ab,\frac cd,\frac ef,\frac gh)$ with 
$\left(M_5(\frac ab,\frac cd,\frac ef,\frac gh),\alpha, \beta, \gamma\right)\in (S^H,T^H,T^{H})$  
and $\{\alpha,\beta,\gamma\}=\{0,1,\infty\}$ is equivalent to one of:
\[
M_5(\tfrac 1n,\tfrac cd,k,\tfrac gh)\qquad\text{(Family 1)}
\qquad\text{ or }\qquad 
M_5(\tfrac 1n,\tfrac{\e+nk}{k},\tfrac ef,\tfrac gh)\qquad\text{(Family 2)}.
\]

\textbf{We first consider the examples from Family 1}. We know that both 0 and 1 correspond 
to $S^H$ or $T^H$--slopes. Examining the 1--slope we obtain 
\[
M_5(\tfrac 1n, \tfrac cd, k, \tfrac gh)(1)\mathop{=}\limits_{(\ref{5CL1})} 
F(\tfrac{1-n}{n}, \tfrac cd, k,
\tfrac{g-h}{h})\mathop{=}\limits_{\textrm{Lemma }\ref{filling_F}} \seifdue D{1-n}nk1\bigu 0110\seifdue Dcd{g-h}h 
\]
which has an essential torus unless $0,\pm1\in\{1-n,k,c,g-h\}$ by
Lemma \ref{l:down}. Since we are interested in hyperbolic manifolds
and instructions not factoring through $M_{4}$, we can use Lemma
\ref{5CL_factor} to rule out the possibilities $1-n,k\in\{0,\pm1\}$
and $c,g-h=0$. We are then left with the cases $c=\pm1$ and $g=h\pm1$.
\begin{description}
\item[\textbf{Case $c=\pm1$}] Turning now our attention to the slope 0 and writing $\tfrac cd=\tfrac 1m$, it holds 
\[
M_5(\tfrac 1n, \tfrac 1m, k, \tfrac gh)(0)\mathop{=}\limits_{(\ref{5CL0})} 
F(\tfrac n{n-1}, 1-m, -\tfrac hg,k-1), 
\]
which, by Lemmata~\ref{l:down} and \ref{filling_F}, has an essential torus unless we are in the case
$0,\pm1\in\{n,1-m,-h,k-1\}$. This time 
Lemma \ref{5CL_factor} leaves us with the necessary condition $h=\pm 1$, which translates to $\frac gh\in\Z$. 
Combining the two necessary conditions and writing $\frac gh=l$, we learn that
\begin{gather*}
M_5(\tfrac 1n, \tfrac 1m, k, l)(1)\mathop{=}\limits_{(\ref{5CL1})} 
F(\tfrac{1-n}n, \tfrac 1m, k, l-1)\mathop{=}\limits_{(\ref{fillingF})} 
\seifdue D{1-n}nk1\bigu 0110\seifdue D1m{l-1}1\\
\mathop{=}\limits_{(\ref{graph_eq1})} \seiftre{S^2}{1-n}nk1 {1+ml-m}{1-l}
\end{gather*}
which is in $S^H\cup T^H$ only if $\pm1\in\{1+m(l-1),1-n,k\}$ by Lemma \ref{l:down}. Lemma \ref{5CL_factor} 
is used to rule out any of these cases occurring. For
instance, if $1+m(l-1)=-1$, then $m(1-l)=2$ and $m\in\{\pm1,\pm2\}$;
moreover, if $m=-2$, then $1-l=-1$ and $l=2$, meaning that it factors
through $M_4$ by Lemma \ref{5CL_factor}.

\item[\textbf{Case $g=h\pm1$}] As before, turning now our attention to
  the slope 0 and writing $\tfrac gh=1+\frac 1m$, we have
\[
M_5(\tfrac 1n, \tfrac cd, k,\tfrac{m+1}m)(0)
\mathop{=}\limits_{(\ref{5CL0})} 
F(\tfrac n{n-1}, \tfrac {c-d}{c}, -\tfrac m{m+1},k-1), 
\]
which, unless $0,\pm1\in\{n,c-d,m,k-1\}$, will have an essential torus
by Lemmata~\ref{l:down} and \ref{filling_F}. Just as in the
preceding case, we can use Lemma~\ref{5CL_factor} to conclude that the
only possibility is $c-d=\pm 1$, which is equivalent to $\tfrac
cd=1+\tfrac 1l$. Combining the necessary conditions, we obtain
that
\begin{gather*}
M_5(\tfrac 1n,\tfrac {l+1}l, k,\tfrac{m+1}m)(1)
\mathop{=}\limits_{(\ref{5CL1})} 
F(\tfrac{1-n}{n}, \tfrac {l+1}l, k,
\tfrac{1}{m})
\mathop{=}\limits_{(\ref{fillingF})} 
\seifdue D{1-n}nk1\bigu 0110\seifdue D{l+1}l1m\\
\mathop{=}\limits_{(\ref{graph_eq1})} \seiftre{S^2}{1-n}nk1 {l+ml+m}{-l-1}
\end{gather*}
which is in $S^H\cup T^H$ only if $\pm1\in\{l+ml+m,1-n,k\}$ by Lemma \ref{l:down}. Lemma \ref{5CL_factor} 
is then used to rule out any of these cases. For instance, if
$l+ml+m=1$, then $m(1+l)=1-l$; if $l=-1$ then $\frac{l+1}{l}=0$,
otherwise $\frac{2}{1+l}-1=m\in\Z$ so $1+l\in\{\pm1,\pm2\}$, that is
$\frac{1+l}{l}\in\big\{\frac23,\frac12,\infty,2\big\}$; moreover, if
$\frac{1+l}{l}=\frac23$, then $m=-2$ and $\frac{m+1}{m}=\frac12$,
meaning that it factors through $M_4$ by Lemma \ref{5CL_factor}. 
The remaining three cases for $\frac{1+l}{l}$ are directly 
excluded by Lemma \ref{5CL_factor}.
\end{description}

\textbf{We consider now the examples from Family 2}. The analysis follows verbatim the steps considered in the study of Family 1. We have assumed that both 0 and 1 correspond 
to $S^H$ or $T^H$--slopes. The manifold $M_5(\tfrac 1n, \tfrac{\e+kn}{k}, \tfrac ef, \tfrac gh)(1)\mathop{=}\limits_{(\ref{5CL1})} 
F(\tfrac{1-n}n,  \tfrac{\e+kn}{k}, \tfrac ef, \tfrac {g-h}h)$
has an essential torus unless 
$0,\pm1\in\{1-n,\e+kn,e,g-h\}$ by Lemmata \ref{l:down} and \ref{filling_F}. 
Lemma \ref{5CL_factor} implies that $1-n,\e+kn\not\in{0,\pm1}$ and $e,g-h\neq 0$. 
We are thus left with the possibilities $e=\pm1$ and $g-h=\pm1$. 

\begin{description} 
\item[\textbf{Case $e=\pm1$}] Writing $\tfrac ef=\tfrac 1m$ we have 
\[
M_5(\tfrac 1n, \tfrac{\e+kn}{k}, \tfrac 1m, \tfrac gh)(0)\mathop{=}\limits_{(\ref{5CL0})} 
F(\tfrac n{n-1}, \tfrac{\e+(n-1)k}{\e+kn},-\tfrac hg,\tfrac{1-m}m)
\]
which has again an essential torus unless 
$0,\pm1\in\{n, \e+(n-1)k,h,1-m\}$. 
Lemma \ref{5CL_factor} leaves us with the necessary condition 
$h=\pm 1$. Indeed, among the other cases, the worst situation is
$\e+(n-1)k=-\e$, but then $(n-1)k=-2\e$ and $n-1\in\{\pm1,\pm2\}$. The
first three cases can directly be ruled out by Lemma \ref{5CL_factor},
and the assumption $n-1=2$ implies that $k=-\e$ and hence that
$\frac{\e+kn}{k}=2$, meaning that it factors through $M_4$ by Lemma \ref{5CL_factor}.
We can hence set $\tfrac gh=l$. This gives
\begin{gather*}
M_5(\tfrac 1n, \tfrac{\e+kn}{k}, \tfrac 1m, l)(1)\mathop{=}\limits_{(\ref{5CL1})} 
F(\tfrac{1-n}n,  \tfrac{\e+kn}{k}, \tfrac 1m, l-1)\mathop{=}\limits_{(\ref{fillingF})} 
\seifdue D{1-n}n1m\bigu 0110\seifdue D{\e+kn}k{l-1}1\\
\mathop{=}\limits_{(\ref{graph_eq1})}  \seiftre{S^2}{n+m(1-n)}{n-1}{l-1}1{\e+nk}k
\end{gather*}
which is in $S^H\cup T^H$ only when $\pm1\in\{n+m(1-n),l-1,\e+nk\}$ by Lemma \ref{l:down}. These cases are all 
discounted using Lemma \ref{5CL_factor}. 

\item [\textbf{Case $g=h\pm1$}] Turning now our attention to the slope 0 and writing $\tfrac gh=1+\frac 1m$, we get
\[
M_5(\tfrac 1n, \tfrac{\e+kn}{k}, \tfrac ef, \tfrac {m+1}m)(0)\mathop{=}\limits_{(\ref{5CL0})} 
F(\tfrac n{n-1}, \tfrac{\e+(n-1)k}{\e+kn},-\tfrac m{m+1},\tfrac{e-f}f) 
\]
which, unless $0,\pm1\in\{n, \e+(n-1)k,m,e-f\}$, will have an essential torus by Lemmata~\ref{l:down} 
and \ref{filling_F}. Once again we use Lemma~\ref{5CL_factor} to
conclude that the only possibility is
$e-f=\pm 1$, which can be reformulated as $\tfrac ef=\tfrac{l+1}l$. Combining the
necessary conditions, we obtain
\begin{gather*}
M_5(\tfrac 1n, \tfrac{\e+kn}{k}, \tfrac {l+1}l, \tfrac {m+1}m)(1)\mathop{=}\limits_{(\ref{5CL1})} 
F(\tfrac{1-n}n,  \tfrac{\e+kn}{k}, \tfrac {l+1}l, \tfrac 1m)\mathop{=}\limits_{(\ref{fillingF})} 
\seifdue D{1-n}n{l+1}l\bigu 0110\seifdue D{\e+kn}k1m\\
\mathop{=}\limits_{(\ref{graph_eq1})}  \seiftre{S^2}{1-n}n{l+1}l{k+m(\e+kn)}{-kn-\e}
\end{gather*}
which is in $S^H\cup T^H$ only if $\pm1\in\{k+m(\e+kn),1-n,l+1\}$ by Lemma \ref{l:down}. Lemma \ref{5CL_factor} 
is used to directly rule out the $\pm1\in\{1-n, l+1\}$ cases. 
Now, if $k+m(\e+kn)=\eta$, then $m=-\frac 1n+\frac{\eta}{\e+kn}+\frac{\e}{n(\e+kn)}$. We can assume that $n,\e+kn\notin\{0,\pm1\}$ otherwise Lemma \ref{5CL_factor} would apply. It follows that $m\in\big[-\frac32,\frac32\big]$ and hence that $m\in\{0,\pm1\}$ which, again, can be ruled out because of  Lemma \ref{5CL_factor}.
\end{description}

We conclude that if $\left(M_5(\mathcal F),\alpha,\beta,\gamma\right)\in(S^H,T^H,T^H)$ 
then $\mathcal F$ factors through $M_4$. This completes Section~\ref{lens_lens5CL_subsec}. 

\subsection{Hyperbolic knots with two lens surgeries arising from the 4--chain link}\label{lens_lens4CL_subsec}
 
In this section we prove that if $M_4(\frac ab,\frac cd,\frac ef)$
is hyperbolic with three fillings in $S^H\cup T^H$ then the instruction 
$(\frac ab,\frac cd,\frac ef)$ factors through $M_3$.
 % \[
 % \dessin{4cm}{4CLf}
 % \]
From \cite[Theorem~5]{excep_slopes} and a careful inspection of
\cite[Tables 12, 21, 22]{excep_slopes} we deduce that if the triple
$\left(M_4(\frac ab,\frac cd,\frac ef),\alpha, \beta,\gamma\right)$ is
in $(S^H,T^H,T^H)$, then $e\left(M_4(\frac ab,\frac cd,\frac ef)\right)=4$ and $\{\alpha,\beta,\gamma\}\subset\{0,1,2,\infty\}$.
Since $\Delta(S^H,T^H)=\Delta(T^H,T^H)=1$ \cite{CGLS}, it follows that
either  $\{\alpha, \beta,\gamma\}=\{1,2,\infty\}$ or $\{\alpha,
\beta,\gamma\}=\{0,1,\infty\}$.
But one can observe that
\begin{eqnarray*}
M_4(\tfrac ab,\tfrac cd,\tfrac ef,\tfrac gh)
& \mathop{\cong}\limits_{\textrm{Lemma }\ref{lem:M5M4}}&
M_5(\tfrac ab,\tfrac{c-d}d,-1,\tfrac{e-f}f,\tfrac gh)
\
                                                         \mathop{\cong}\limits_{(\ref{symeq9})\circ(\ref{linksymeq1})^2} \
M_5(\tfrac{c-2d}{c-d},\tfrac{b}{b-a},-1,\tfrac{h}{h-g},\tfrac{e-2f}{e-f})\\
&\mathop{\cong}\limits_{\textrm{Lemma }\ref{lem:M5M4}}&
M_4(\tfrac{c-2d}{c-d},\tfrac{2b-a}{b-a},\tfrac{2h-g}{h-g},\tfrac{e-2f}{e-f})
\ \mathop{\cong}\limits_{\textrm{Lemma }\ref{4CL_symmetry}}\ 
M_4(\tfrac{e-2f}{e-f},\tfrac{c-2d}{c-d},\tfrac{2b-a}{b-a},\tfrac{2h-g}{h-g}).
\end{eqnarray*}
It follows that $\left(M_4(\tfrac ab,\tfrac cd,\tfrac ef),0,1,\infty\right)\cong\left(M_4(\tfrac{e-2f}{e-f},\tfrac{c-2d}{c-d},\tfrac{2b-a}{b-a}),2,\infty,1\right)$.
It is hence sufficient to study the case $\{\alpha, \beta,\gamma\}=\{1,2,\infty\}$.
We examine now the necessary conditions on the filling instruction
$(\frac ab,\frac cd, \frac ef)$ imposed from $1$,$2$ and $\infty$ being $S^H\cup T^H$--slopes.

\subsubsection{Necessary conditions from $M_{4}\big(\tfrac ab,\tfrac cd,\tfrac ef\big)(2)$}\label{sec:equations}
We have
\[
M_4(\tfrac ab,\tfrac cd,\tfrac ef)(2)
\mathop{=}\limits_{(\ref{4CL2})}
\seifdue{D}{a-b}b{e-f}f \bigu0110 \seifdue{D}cd2{-1}
\]
which is in $S^H\cup T^H$ only if $0,\pm1\in\{a-b, e-f, c\}$ by Lemma \ref{l:down}.
If $a-b=0$, $e-f=0$, $c=0$ then $\frac{a}{b}=1$, $\frac{e}{f}=1$, $\frac{c}{d}=0$ 
respectively, which are all excluded by Lemma~\ref{4CL_factor} since we are only interested in the hyperbolic case. We continue with a case by case analysis:

\begin{description}

\item[Case $|a-b|=1$] Up to a simultaneous change of signs 
for $a$ and $b$, we may assume that $a-b=1$. This gives us, again by Lemma~\ref{l:down},
\begin{gather*}
M_4\big(\tfrac ab,\tfrac cd,\tfrac ef\big)(2)
\mathop{=}\limits_{(\ref{4CL2})}
\seifdue{D}{1}b{e-f}f \bigu0110 \seifdue{D}cd2{-1}
 \mathop{=}\limits_{(\ref{graph_eq1})}\seiftre{S^2}cd2{-1}{f+b(e-f)}{f-e}
\end{gather*}
which is in $S^H\cup T^H$ only if $\pm1\in\{c, f+b(e-f)\}$. 
Up to changing the signs of $c$ and $d$ or of
$e$ and $f$, we may hence assume that
either $c=1$ or $b(f-e)=1+f$.

\item[Case $|e-f|=1$] Lemma \ref{4CL_symmetry} tells us that 
$M_4(\tfrac ab,\tfrac cd,\tfrac ef)(2)=M_4(\tfrac ef,\tfrac cd,\tfrac ab)(2)$. So, any example 
found in this case is contained in the case $|a-b|=1$;

\item[Case $|c=1|$] Up to a simultaneous change of signs 
for  $c$ and $d$, we may 
assume that $c=1$. We get
\begin{gather*}
M_4(\tfrac ab,\tfrac cd,\tfrac ef)(2)
\mathop{=}\limits_{(\ref{4CL2})}
\seifdue{D}{a-b}b{e-f}f \bigu0110 \seifdue{D}1d2{-1}
\\ \mathop{=}\limits_{(\ref{graph_eq1})}
\seiftre{S^2}{e-f}f{a-b}b{1-2d}2
\end{gather*}
which is in $S^H\cup T^H$ only when $\pm1\in\{a-b, e-f, 1-2d\}$ by Lemma \ref{l:down}.
If
$1-2d=\pm1$, then $d\in\{0,1\}$, that is $\frac{c}{d}\in\{1,\infty\}$, which is excluded by 
Lemma~\ref{4CL_factor}. 
 So either $|a-b|=1$ or $|e-f|$ equals 1, and we are left with one of
the previous cases.
\end{description}

To summarise, if $M_4(\tfrac ab,\tfrac cd,\tfrac ef)(2)\in S^H\cup T^H$, then 
one of the following sets of conditions holds:
\[
\fbox{$
  \begin{array}{ll}
a-b=1&(\textrm{\CDZ})\\[.2cm]
c=1&(\textrm{\CDP})
\end{array}$}\ (\textrm{\CDU})
\ \ \textrm{ or }\ \ 
\fbox{$
  \begin{array}{ll}
a-b=1&(\textrm{\CDZ})\\[.2cm]
b(f-e)=1+f&(\textrm{\CDS})
\end{array}$}\ (\textrm{\CDD}).
\]

\subsubsection{Necessary conditions from $M_4(\tfrac ab,\tfrac cd,\tfrac ef)(1)$}
We have
\[
M_4\big(\tfrac ab,\tfrac cd,\tfrac ef\big)(1)\mathop{=}\limits_{(\ref{4CL1})}
\seiftre{S^2}{a-2b}b{c-d}c{e-2f}f
\]
which is in $S^H\cup T^H$ only if $\pm1\in\{a-2b, c-d, e-2f\}$ by Lemma \ref{l:down}.
So, one of the 
following conditions necessarily holds:
\[
\fbox{$a-2b=\e_1$}\ (\textrm{\CUU})
\ \ \textrm{ or }\ \ \fbox{$c-d=\e_1$}\ (\textrm{\CUD})
\ \ \textrm{ or }\ \ \fbox{$e-2f=\e_1$}\ (\textrm{\CUT}).
\]

\subsubsection{Necessary conditions from  $M_4(\tfrac ab,\tfrac cd,\tfrac ef)(\infty)$}
We have 
\[
M_4(\tfrac ab,\tfrac cd,\tfrac ef)(\infty)
\mathop{=}\limits_{(\ref{4CLinf})}
\big(S^2,(a,b),(d,-c),(e,f)\big)
\]
which is in $S^H\cup T^H$ only when $\pm1\in\{a, d, e\}$ by Lemma \ref{l:down}. So, one of the following conditions necessarily holds:
\[
\fbox{$a=\e_\infty$}\ (\textrm{\CIU})
\ \ \textrm{ or }\ \ \fbox{$d=\e_\infty$}\ (\textrm{\CID})
\ \ \textrm{ or }\ \ \fbox{$e=\e_\infty$}\ (\textrm{\CIT}).
\]

\subsubsection{Enumeration of $M_4(\tfrac ab,\tfrac cd,\tfrac ef)$ satisfying the necessary conditions}\label{sec:enumeration}
We have shown that if the triple
$(M_4(\mathcal{F}),\alpha,\beta,\gamma)$ is in $(S^H,T^H,T^H)$ then $\mathcal{F}$ is equivalent to a filling instruction $(\tfrac ab,\tfrac cd,\tfrac ef)$ satisfying 
one of \CDU\ or \CDD, one of \CUU, \CUD\ or \CUT\, and one of \CIU, \CID\ or \CIT. We will now show 
that any such $(\tfrac ab,\tfrac cd,\tfrac ef)$ must factor through
$M_3$. First, we begin by emphasizing a few incompatibilities between
the above conditions.
\begin{description}
\item[\CDZ$\ +\ $\CUU] substituting $a-b=1$ into $a-2b=\e_1$ gives
  $\frac{a}{b}=1+\frac{1}{1-\e_1}\in\big\{\frac{3}{2},\infty\big\}$, which is excluded by 
Lemma \ref{4CL_factor}.
\item[\CDZ$\ +\ $\CIU] substituting $a=\e_\infty$ into $a-b=1$ gives 
$\frac{a}{b}=\frac{\e_\infty}{\e_\infty-1}=1+\frac{1}{\e_\infty-1}\in\big\{\frac{1}{2},\infty\big\}$
which is excluded by Lemma \ref{4CL_factor}.

\saut

\item[\CDP$\ +\ $\CUD] substituting $c=1$ into $c-d=\varepsilon$ gives $\frac{c}{d}=\frac{1}{1-\e_1}\in\big\{\frac{1}{2},\infty\big\}$, which is excluded by 
Lemma \ref{4CL_factor}.
\item[\CDP$\ +\ $\CID] this gives $\frac{c}{d}=\pm1$, which is excluded by Lemma \ref{4CL_factor}.

\saut

\item[\CDD$\ +\ $\CUT] \CUT\, implies $e-f=f+\e_1$ which we substitute into $b(f-e)=1+f$ 
to get $-b(f+\e_1)=1+f$. If $f=-\e_1$ then $\frac ef=-1$ which is excluded by Lemma \ref{4CL_factor}, 
and otherwise $-b=1+\frac{1-\e_1}{f+\e_1}$.
  If $\e_1=1$, then $b=-1$, $a=0$ and $\frac{a}{b}=0$ which are excluded by Lemma \ref{4CL_factor}.
  If $\e_1=-1$ then $\frac{2}{f-1}=-b-1$ is an integer, so $f-1$ divides
  2 and $f\in\{-1,0,1,2,3\}$. If $f=3$, then $b=-2$, $a=-1$ and
  $\frac{a}{b}=\frac{1}{2}$ which is excluded by Lemma \ref{4CL_factor}.
  Otherwise,
  $\frac{e}{f}=2-\frac{1}{f}\in\big\{1,\frac{3}{2},3,\infty\big\}$ which is excluded by 
  Lemma \ref{4CL_factor};
\item[\CDD$\ +\ $\CIT] if $e=\e_\infty$ then $f\neq\pm1$ by Lemma \ref{4CL_factor}. 
Substituting $e=\e_\infty$ into $b(f-e)=1+f$ gives $b=1+\frac{1+\e_\infty}{f-\e_\infty}$.
  If $\e_\infty=-1$, then $b=1$,  $a=2$ and $\frac{a}{b}=2$ which is excluded by Lemma \ref{4CL_factor}.
  If $\e_\infty=1$ then $\frac{2}{f-1}=b-1$ is an integer, and $f-1$ divides
  2 which implies $f\in\{-1,0,1,2,3\}$. If $f=3$, then $b=2$, $a=3$ and
  $\frac{a}{b}=\frac{3}{2}$ which is excluded by Lemma \ref{4CL_factor}.
  Otherwise, $\frac{e}{f}=\frac{1}{f}\in\big\{-1,\frac{1}{2},1,\infty\big\}$ which is excluded by Lemma 
  \ref{4CL_factor}.
  
\saut

\item[\CUD$\ +\ $\CID] we have 
$c=\e_\infty+\e_1$ and hence $\frac{c}{d}=\frac{\e_\infty+\varepsilon}{\e_\infty}=1+\e_1\e_\infty\in\{0,2\}$, which is excluded by 
Lemma \ref{4CL_factor}.
\item[\CUT$\ +\ $\CIT] we have $2f=\e_\infty-\e_1\Rightarrow f\in\{0,\pm1\}$, so 
$\frac{e}{f}\in\{\pm1,\infty\}$ which is excluded by Lemma \ref{4CL_factor}. 
\end{description}

\saut

We now observe that the above analysis is enough to conclude:
\begin{itemize}
\item \CDZ\ necessarily holds. The above analysis implies that \CUU\ or \CIU\ do not hold.
\item If \CDU\ holds then, because of
  \CDP, neither \CUD\ or \CID\ hold. It follows that both
\CUT\ and \CIT\ hold, but they can't hold simultaneously. So \CDD\ holds. 
\item If \CDD\ holds then neither \CUT\ or \CIT\ can
  hold. It follows that both \CUD\ and \CID\ hold, but they can't hold
  simultaneously.
\end{itemize}

We conclude that, as announced, if $(M_4(\mathcal{F}),\alpha,\beta,\gamma)\in(S^H,T^H,T^H)$ then 
$\mathcal{F}$ factors through $M_3$.

\subsection{Hyperbolic knots with two lens surgeries arising from the 3--chain link}\label{lens_lens3CL_subsec}

We now enumerate all the hyperbolic knots with two lens space surgeries obtained by surgery 
on the 3--chain link. We prove the following result:

\begin{proposition} If $\left(M_3(\frac ab,\frac cd),\alpha,\beta,\gamma\right)\in(S^H,T^H,T^H)$ 
then it is
equivalent to $\left( N\big(-\tfrac{5}{2},\tfrac{1-2k}{5k-2}\big),\infty,-2,-1\right)$, for some $k\in\Z$, or to
$\left(N\big(-\frac{3}{2},-\frac{14}{5}), -2,-1,\infty\right)$.
\end{proposition}

The enumeration of all $(S^H,T^H,T^H)$ triples obtained by surgery on \tCL\ comes from 
\cite[Theorem~1.3]{Magic} and a careful examination of \cite[Tables~2--3]{Magic}. It should be noted that the classification 
of exceptional fillings on the exterior of the 3-chain link in \cite{Magic} is performed on the 
exterior of the mirror image $\textrm{\tCL}^{\ast}$. The exterior of
$\textrm{\tCL}^{\ast}$ is denoted by $N$, 
and, of course, $M_3(\tfrac ab,\frac cd,\frac ef)=N(-\frac ab,-\frac cd,-\frac ef)$. 
For the sake of clarity when referencing to tables, we will adopt the convention in \cite{Magic}. 

First, we note that \cite[Table~4]{Magic} 
involves no fillings of the form $\lens{\star}{\star}$ so we restrict our attention to  
\cite[Tables~2--3]{Magic}. In Table~3, there are some entries where
$N(\tfrac pq,\tfrac rs,\tfrac tu)=\lens{\star}{\star}$; however, in
each case, slopes $-1$ or 0 are involved so, if $\left(N(\frac ab,\frac
  cd),\alpha,\beta,\gamma\right)$ was a $(S^H,T^H,T^H)$ triple which
had some entries in Table 3, then $\beta$ or $\gamma$ would be the
$-1$ or 0 slope, otherwise $N(\frac ab,\frac cd)$ would not be
hyperbolic, and then there would be no value for $\alpha$ to carry an $S^H$--surgery. 
It follows then from \cite[Theorem~1.3]{Magic} that, up to equivalency, if 
$\left(N(\frac ab,\frac
  cd),\alpha,\beta,\gamma\right)\in(S^H,T^H,T^H)$ then we can assume
that at least one of the slopes is $-3$, $-2$, $-1$, 0 or $\infty$ and
that, because of Lemma \ref{N_factor},
$\alpha,\beta,\gamma\in\{-3,-2,-1,0,\infty\}$, otherwise $M_3(\frac ab,\frac
  cd)$ would not be hyperbolic. In particular the $S^H$--slope $\alpha$ is in $\{-3,-2,-1,0,\infty\}$. We now examine each case individually.

\subsubsection{Case 0 is an $S^H$--slope}

We see directly from \cite[Table~2]{Magic} that if 
$N(\frac rs,\frac tu)(0)=\lens{\star}{\star}$ then $\tfrac rs=n$, $\tfrac tu = -4-n+\tfrac 1m$ and 
$N(\frac rs,\frac tu)(0)=\lens{6m-1}{2m-1}$. So, if $N(\frac rs,\frac tu)(0)=S^3$ then $m=0$ and 
$\frac{t}{u}=\infty$, which is discarded by Lemma \ref{N_factor}. 

\subsubsection{Case $-1$ is an $S^H$--slope}

We see directly from \cite[Table~2]{Magic} that if 
$N(\frac rs,\frac tu)(-1)=\lens{\star}{\star}$ then $\tfrac rs=-3+\tfrac1n$, and 
$N(\frac rs,\frac tu)(-1)=\lens{2n(t+3u)-t-u}{\star}$. If 
$\lens{2n(t+3u)-t-u}{\star}=S^3$ then $2n(t+3u)-t-u=\pm1$. 
By changing the signs of both $t$ and $u$, we may assume w.l.o.g.\ that
\begin{equation}
2n(t+3u)-t-u=1.\label{eq:-1_S3}
\end{equation}
Moreover, we know by \cite{CGLS} that $\Delta(S^H,T^H)=1$ and since
$\Delta(-3,-1)=2$, it follows that
$\beta,\gamma\in\{-2,0,\infty\}$. But, by \cite{CGLS}, we also know that
$\Delta(T^H,T^H)=1$, so the only pairs of possibilities for the
$T^H$--slopes are $\{-2,\infty\}$ and $\{0,\infty\}$.
From \cite[Theorem~1.3]{Magic} we know that $N(\frac rs,\frac tu)(\infty)$ is always a lens space.

We will now use (see \cite[Table 2]{Magic}) to further refine the constraints, $\tfrac rs=-3+\tfrac1n$ and \eqref{eq:-1_S3}, that we have found from imposing $-1$ to be a $S^{3}$--slope. This time we will analyse the restrictions we obtain by considering $0$ and $-2$ to be lens space slopes and through this analysis we will enumerate all $(S^H,T^H,T^H)$ 
triples. We will denote the new parameters with primes.

\begin{description}
\item[Case $0$ is a $T^H$--slope]
Either $\tfrac rs=-3+\frac{1}{n}=n'$ or $\tfrac rs=-3+\frac{1}{n}=-4-n'+\frac{1}{m'}$.

\begin{description}
\item[Case $\tfrac rs=-3+\frac{1}{n}=n'$]
Then $n'=-2$ or $-4$, and $\tfrac tu=-4-n'+\tfrac{1}{m'}$. The case $n'=-2$ is excluded by Lemma \ref{N_factor}. 
The case $n'=-4$ implies that $n=-1$ and that
$\tfrac{t}{u}=\tfrac{1}{m'}$, that is $u=m't$. From (\ref{eq:-1_S3}),
we have then 
$-t(3+7m')=1$ which cannot hold.

\item[Case $\tfrac rs=-3+\frac{1}{n}=-4-n'+\frac{1}{m'}$]
Then $n'+1=\frac{1}{m'}-\frac{1}{n}\in[-2,2]\cup\{\infty\}$ so
$n'\in\{-3,-2,-1,0,1,\infty\}$. But $\frac{t}{u}=n'$ and, by
Lemma~\ref{N_factor}, we know that $n'=1$ otherwise $N(\tfrac
rs,\tfrac tu)$ would not be hyperbolic. It follows that $n=-1$ and
substituting this information in \eqref{eq:-1_S3} we obtain $-10u=1$ which cannot hold.
\end{description}

\item[Case $-2$ is a $T^H$--slope]
Either $\tfrac rs=-3+\frac{1}{n}=-2+\frac{1}{n'}$ or $-2+\frac{1}{n'}=\frac{t}{u}$.
\begin{description}
\item[Case $\tfrac rs=-2+\frac{1}{n'}=-3+\frac{1}{n}$]
Then $n=2$ and by (\ref{eq:-1_S3}), $\left(N(-\frac{5}{2},\frac{t}{u}),-1,-2,\infty\right)$ is a
$(S^H,T^H,T^H)$ triple whenever $3t+11u=1$. Namely, for $t=4-11k$ and
$u=3k-1$ with any $k\in\Z$. That is
\fbox{$\left(N(-\tfrac{5}{2},\tfrac{4-11k}{3k-1}), -1,-2,\infty\right)$}
are $(S^H,T^H,T^H)$ triples for every $k\in\Z$.

\item[Case $-2+\frac{1}{n'}=\frac{t}{u}$]
Then $\frac{t}{u}=\frac{1-2n'}{n'}$ so (\ref{eq:-1_S3}) becomes
$2n(1+n')+n'\in\{0,2\}$.

If $2n(1+n')+n'=0$, then
$2n=\frac{1}{1+n'}-1\in[-2,0]$ so
$(n,n')\in\big\{(-1,-2),(0,0)\big\}$. The first case leads to the
$(S^H,T^H,T^H)$ triple
\fbox{$\left(N(-4,-\frac{5}{2}),-1,-2,\infty\right)$} whereas the
second is discarded by Lemma~\ref{N_factor} since $\frac rs=\frac tu=\infty$. 

If $2n(1+n')+n'=2$, then $\frac{3}{2n+1}=n'+1\in\Z$. It follows that
$n\in\{-2,-1,0,1\}$.
For $n\in\{0,1\}$, we have $\frac{r}{s}=-3+\frac{1}{n}\in\{-2,\infty\}$ so the
associated space is non-hyperbolic by Lemma \ref{N_factor}. For $n=-2$ and 
$-1$ we find that \fbox{$\left(N(-\frac{7}{2},-\frac{5}{2}),-1,-2,\infty\right)$} and
\fbox{$\left(N(-4,-\frac{9}{4}),-1,-2,\infty\right)$} are $(S^H,T^H,T^H)$ triples.
\end{description}
\end{description}

\subsubsection{Case $-2$ is an $S^H$--slope}

We see directly from \cite[Table~2]{Magic} that if 
$N(\frac rs,\frac tu)(-2)=\lens{\star}{\star}$ with $N(\frac rs,\frac tu)$ hyperbolic 
then $\tfrac rs=-2+\tfrac1n$, and $N(\frac rs,\frac tu)(-2)=\lens{(3n(t+2u)-2t-u}{\star}$.
So, up to simultaneously reversing the signs of $t$ and $u$, we may assume
w.l.o.g. that
\begin{equation}
3n(t+2u)-2t-u=1.\label{eq:-2_S3}
\end{equation}

As in the previous section, since $\Delta(S^H,T^H)=\Delta(T^H,T^H)=1$ the 
only possible pairs of $T^H$--slopes are $\{-3,\infty\}$ and $\{-1,\infty\}$.
We know that the $\infty$--filling is always a lens space by
\cite[Theorem~1.3]{Magic}.
We now enumerate the new conditions arising from $-1$ or $-3$ being
$T^H$--slopes. We will denote the new parameters with primes.

\begin{description}
\item[Case $-1$ is a $T^H$--slope]

From \cite[Table~2]{Magic}, either $-2+\frac{1}{n}=-3+\frac{1}{n'}$ or
$\frac{t}{u}=-3+\frac{1}{n'}$.

\begin{description}
\item[Case $-2+\frac{1}{n}=-3+\frac{1}{n'}$]
Then $n=-2$ and we find that 
$\left(N(-\frac{5}{2},\frac{t}{u}),-1,-2,\infty\right)$ is a
$(S^H,T^H,T^H)$ triple whenever $8t+13u+1=0$, that is for $t=13k-5$ and
$u=3-8k$ with any $k\in\Z$. So
\fbox{$\left(N(-\frac{5}{2},\frac{13k-5}{3-8k}),-2,-1,\infty\right)$}
is a $(S^H,T^H,T^H)$ triple for every $k\in\Z$.

\item[Case $-3+\frac{1}{n'}=\frac{t}{u}$]
In this case $\frac{t}{u}=\frac{1-3n'}{n'}$ and (\ref{eq:-2_S3}) becomes
$3n(1-n')+5n'\in\{1,3\}$.

If $3n(1-n')+5n'=1$ then $\frac{4}{5-3n}=1-n'\in\Z$. It follows that
$n\in\{1,2,3\}$. For $n\in\{2,3\}$ we find that 
\fbox{$\left(N(-\frac{3}{2},-\frac{14}{5}),-2,-1,\infty\right)$} and \fbox{$\left(N(-\frac{5}{3},-\frac{5}{2}),-2,-1,\infty\right)$} are
$(S^H,T^H,T^H)$ triples.

If $3n(1-n')+5n'=3$ then $\frac{2}{5-3n}=1-n'\in\Z$. It follows that
$n\in\{1,2\}$.
For $n=1$, we have $\frac{r}{s}=-1$ which makes $N(\tfrac rs,\tfrac tu)$ 
non-hyperbolic by Lemma \ref{N_factor}. 
For $n=2$ we find that 
\fbox{$\left(N(-\frac{3}{2},-\frac{8}{3}),-2,-1,\infty\right)$} is a
$(S^H,T^H,T^H)$ triple.
\end{description}

\item[Case $-3$ is a $T^H$--slope]

If $\frac rs$ or $\frac tu$ is $-2$ then $N(\tfrac rs,\tfrac tu)$ is non-hyperbolic by
Lemma \ref{N_factor}.
So, from \cite[Table~2]{Magic}, $-2+\frac{1}{n}=-1+\frac{1}{n'}$ making 
$n=2$ and $\frac{t}{u}=-1+\frac{1}{m'}=\frac{1-m'}{m'}$.
Using (\ref{eq:-2_S3}), we obtain $m'\in\{-\frac 57,-\frac{3}{7}\}$
which is not an integer.
\end{description}

\subsubsection{Case $-3$ is an $S^H$--slope}

From \cite[Table~2]{Magic}, if $N(\frac rs,\frac tu)(-3)=\lens{\star}{\star}$ then either
$\frac{t}{u}=-2$, which is excluded by Lemma \ref{N_factor}, or 
$\frac{r}{s}=-1+\frac{1}{n}$ and $\frac{t}{u}=-1+\frac{1}{m}$. In 
the latter case we have $N(-1+\tfrac1n,-1+\tfrac1m)(-3)=\lens{(2n+1)(2m+1)-4}{\star}=S^{3}$ 
if and only if $(2n+1)(2m+1)-4=\pm1$; that is $(2n+1)(2m+1)=3$ or $5$. Since both 3
and 5 are primes, it follows
that either $2n+1$ or $2m+1$ is $\pm1$. By symmetry, we may assume
that $2n+1=\pm1$, making $n=-1$ or 0 which are both excluded by Lemma \ref{N_factor}.

\subsubsection{Case $\infty$ is an $S^H$--slope}

From \cite[Theorem~1.3]{Magic}, $N(\tfrac rs,\tfrac tu)(\infty)=\lens{tr-us}{\star}$. 
So, $\infty$ is an $S^H$--slope if and only if
\begin{equation}
tr-us=\pm1.
\label{eq:Infty_S3}
\end{equation}

As before, we have $\Delta(S^H,T^H)=\Delta(T^H,T^H)=1$ so the only possible pairs of $T^H$ 
slopes are $\{-3,-2\}$, $\{-2,-1\}$ or $\{-1,0\}$. Each $T^H$--slope imposes conditions on 
$\tfrac rs$, $\tfrac tu$. We will use primes on the parameters to denote the conditions imposed 
from the smallest $T^H$--slope and double primes on the parameters coming from the conditions 
on the second $T^H$--slope.

\begin{description}
\item[Case $0$ is a $T^H$--slope]

From \cite[Table~2]{Magic} we have $\frac{r}{s}=n'$ and
$\frac{t}{u}=-4-n'+\frac{1}{m'}=\frac{1-m'(n' +4)}{m'}$.
Equation \eqref{eq:Infty_S3} becomes then $\big(1-m'(n'+4)\big)n'=m'\pm1$.
According to Lemma~\ref{lem:QuadEq}, we have 
$n'\in\{-5,-4,-3,-2,-1,0,1\}$.
For $N(\tfrac rs,\tfrac tu)$ to be hyperbolic, $n'$ cannot be in $\{-3,-2,-1,0\}$ 
because of Lemma \ref{N_factor}; we are hence left with cases 
$(n',m')\in\big\{(-5,-1),(-4,-5),(-4,-3),(1,0)\big\}$. If $(n',m')=(-5,-1)$ then 
$\tfrac tu=0$, and if $(n',m')=(1,0)$ then $\tfrac tu =\infty$. 
So, these cases are both excluded by Lemma \ref{N_factor}. The other 
two cases 
\fbox{$\left(N(-4,-\frac{1}{5}),\infty,-1,0,\right)$}
and \fbox{$\left(N(-4,-\frac{1}{3}),\infty,-1,0\right)$}
are indeed $(S^H,T^H,T^H)$ triples. 

\item[Case $-2$ and $-1$ are the $T^H$--slopes]

In this case, either $-2+\frac{1}{n'}=-3+\frac{1}{n''}$ or, up to symmetry, 
$\left(\frac{r}{s},\frac{t}{u}\right)=\left(\frac{1-2n'}{n'},\frac{1-3n''}{n''}\right)$.

\begin{description}
\item[Case $-2+\frac{1}{n'}=-3+\frac{1}{n''}$]
Then $n'=-2$ and up to symmetry, we may assume
that $\frac{r}{s}=-\frac{5}{2}$. Up to a simultaneous change of sign for 
$t$ and $u$, equation (\ref{eq:Infty_S3}) becomes $5t+2u=1$ and this
leads to \fbox{$\left(N(-\frac{5}{2},\frac{1-2k}{5k-2}),\infty,-2,-1)\right)$}
which is indeed a $(S^H,T^H,T^H)$ triple for every $k\in\Z$.
\item[Case $\left(\frac{r}{s},\frac{t}{u}\right)=\left(\frac{1-2n'}{n'},\frac{1-3n''}{n''}\right)$]
Equation \eqref{eq:Infty_S3} becomes $2n'+3n''-5n'n''\in\{0,2\}$.

If $2n'+3n''=5n'n''$, then $n'=\frac{3n''}{5n''-2}\in\Z$. 
If $n''\geq0$ then the condition $5n''-2\leq3n''$ implies that
$n''\leq 1$. If 
$n''\leq0$ then the condition $3n''\leq5n''-2$ implies that $n''\geq1$. It follows that 
either $n''=0$, and then $\tfrac tu=\infty$, or $n''=1$, and then
$\tfrac tu=-2$. Both cases are excluded by Lemma~\ref{N_factor}.

If $2n'+3n''=2+5n'n''$, then $n'=\frac{3n''-2}{5n''-2}\in\Z$. 
If $n''<0$ then $0<n'<1$, and if $n''\geq \tfrac23$ then $0<n'<1$. So 
$n''=0$ and $\frac{t}{u}=\infty$ which is excluded by Lemma \ref{N_factor}.
\end{description}
\item[Case $-3$ is a $T^H$--slope]
From \cite[Table~2]{Magic} we have $\frac{r}{s}=-2$, which is excluded
by Lemma \ref{N_factor}, 
or $\frac{r}{s}=\frac{1-n'}{n'}$ and $\frac{t}{u}=\frac{1-m'}{m'}$.
In the latter case (\ref{eq:Infty_S3}) becomes $n'+m'=0$ or 2.
Using $\Delta(T^H,T^H)=1$, if $-3$ is a $T^H$--slope then $-2$ is the only possible second $T^H$. 
But then, according to \cite[Table~2]{Magic}, we have
$-1+\frac{1}{n'}=-2+\frac{1}{n''}$ and hence$n'=-2$. Subbing this value 
into (\ref{eq:Infty_S3}) with $\frac{t}{u}=\frac{1-m'}{m'}$ we find
that either $m'=2$ or 4. This leads to
\fbox{$\left(N(-\frac{3}{2},-\frac{1}{2}),\infty,-3,-2\right)$} and
\fbox{$\left(N(-\frac{3}{2},-\frac{3}{4}),\infty,-3,-2\right)$} which
are actually $(S^H,T^H,T^H)$ triples.
\end{description}

\subsubsection{Identifying cases}
In the above analysis we have proved that the only $(S^H,T^H,T^H)$ triples of the form 
$\left(N(\tfrac rs,\tfrac tu),\alpha,\beta, \gamma\right)$ are:
\begin{multicols}{2}
  \begin{enumerate}[(i)]
  \item\label{case1}
    $A_n:=
    \left(N(-\frac{5}{2},\frac{1-2n}{5n-2}),\infty,-2,-1\right)$
    for $n\in\Z$;
  \item\label{case2}
    $A'_n:=
    \left(N(-\frac{5}{2},\frac{4-11n}{3n-1}),-1,-2,\infty\right)$
    for $n\in\Z$;
  \item\label{case3}
    $A''_n:=
    \left(N(-\frac{5}{2},\frac{13n-5}{3-8n}),-2,-1,\infty\right)$
    for $n\in\Z$;
 \item\label{case4} $\left(N(-4,-\frac{5}{2}),-1,-2,\infty\right)$;
  \item\label{case5} $\left(N(-\frac{7}{2},-\frac{5}{2}),-1,-2,\infty\right)$;
  \item\label{case6} $\left(N(-4,-\frac{9}{4}),-1,-2,\infty\right)$;
 \item\label{case7} $\left(N(-\frac{3}{2},-\frac{14}{5}),-2,-1,\infty\right)$;
 \item\label{case8} $\left(N(-\frac{5}{3},-\frac{5}{2}),-2,-1,\infty\right)$;
  \item\label{case9} $\left(N(-\frac{3}{2},-\frac{8}{3}),-2,-1,\infty\right)$;
  \item\label{case10} $\left(N(-4,-\frac{1}{5}),\infty,-1,0\right)$;
  \item\label{case11} $\left(N(-4,-\frac{1}{3}),\infty,-1,0\right)$;
  \item\label{case12} $\left(N(-\frac{3}{2},-\frac{1}{2}),\infty,-3,-2\right)$;
  \item\label{case13} $\left(N(-\frac{3}{2},-\frac{3}{4}),\infty,-3,-2\right)$.
  \end{enumerate}
\end{multicols}

In this list there are many repetitions. Indeed, using the first
equality in \cite[Theorem~1.5]{Magic}, we obtain that cases
(\ref{case4}) and (\ref{case9}), cases (\ref{case6}) and
(\ref{case7}), cases (\ref{case10}) and (\ref{case13}), and cases
(\ref{case11}) and (\ref{case12}) are pairwise isomorphic. Using the third
equality in \cite[Theorem~1.5]{Magic}, we obtain that cases (\ref{case7}) and (\ref{case13}) on the one hand, and cases
(\ref{case9}) and (\ref{case12}) on the other hand, are pairwise isomorphic. Moreover, using the second equality in \cite[Theorem~1.5]{Magic}, we 
see that $A''_n\cong A'_n\cong A_n$ for every $n\in\Z$. Finally, it can
be noted that, up to Lemma \ref{3CL_sym}, case (\ref{case4}) is $A'_0$,
case (\ref{case4}) is $A'_1$ and case (\ref{case8}) is $A''_0$. Summing
up, all cases are either isomorphic to case (\ref{case7}) or to $A_n$
for some $n\in\Z$.

\subsubsection{Distinctness of examples}

The Berge manifold is the unique hyperbolic knot exterior in a solid torus $T$ with three 
distinct solid torus fillings \cite{Gabai}. The Berge manifold is equal to $N(-\frac52)$ 
\cite{Magic}. By filling along a $\frac1n$--slope on $\partial T$ we obtain a family 
of hyperbolic knot exteriors with two lens space fillings. As our enumeration of 
$(S^H,T^H,T^H)$ triples obtained by surgery on 5CL is exhaustive, the family of 
$(S^H,T^H,T^H)$ triples obtained by filling along a boundary component of the Berge 
manifold is
$\left\{\Big(N(-\frac52,\frac{1-2n}{5n-2}),\infty,-2,-1\Big)\right\}$.

By considering the sets of exceptional fillings, we will now show that 
$N(-\frac32,-\frac{14}5)\neq N(-\frac52,\frac{1-2n}{5n-2})$ for any $n$.
Using \cite[Tables~2--3]{Magic} we can write down the set of exceptional slopes and fillings 
of $N(-\frac52,\frac{1-2n}{5n-2})$ and $N(-\frac32,-\frac{14}5)$; the result is shown in Table 
\ref{not_berge_example}. We immediately observe that $N(-\frac52,\frac{1-2n}{5n-2})$ has three 
distinct toroidal fillings, and that $N(-\frac32,-\frac{14}5)$ has only two 
toroidal filling. This shows $N(-\frac32,-\frac{14}5)\neq N(-\frac52,\frac{1-2n}{5n-2})$ for any $n\in\Z$.

\section{$(S^H,T^H,T$) triples}\label{lens_toroidal_sec}

In this section, we enumerate all $(S^H,T^H,T)$ triples
obtained by surgery on the 5CL and realizing the maximal distance. We know, from  \cite[Theorem~1]{excep_slopes}, that if $(M_{5}(\mathcal F),\beta,\gamma)\in(T^H,T)$, then $\Delta(\beta,\gamma)\leq3$.
\begin{theorem}\label{lens_toroidal_prop}
\begin{itemize}
\item[]
\item If $\left(M_{5}(\mathcal F),\alpha,\beta,\gamma\right)\in(S^H,T^H,T)$ with $\Delta(\beta,\gamma)=3$, 
then it is equivalent to either the triple $B_n:=\left(N(-1+\tfrac{1}{n},-1-\tfrac{1}{n}), \infty,-3,0\right)$ for some integer
$n\geq 2$, or to $C_n :=\left(N(-1+\tfrac{1}{n},-1-\tfrac{1}{n-2}), \infty,-3,0\right)$
for some integer $n\geq 4$.
\item For $n> 2$,
$E(B_n)=\{-3,-2,-1,0,\infty\}$ and the exceptional fillings are 
given in Table~\ref{te}. For $n=2$, $B_{n}$ is the exterior of the pretzel knot $(-2,3,7)$ and $e(B_{n})=7$.
\item For $n\geq 4$,
  $E(C_n)=\{-3,-2,-1,0,\infty\}$ and the exceptional fillings are
  given in Table~\ref{te}. 
\item None of the $B_n$ is equivalent to a $C_k$.
\end{itemize}
\end{theorem}

\begin{table}
\begin{center}
\begin{tabular}{ |c|c| }
\hline
\multicolumn{2}{ |c| }{$n>2$, 
\qquad $E(N(-1+\frac1n,-1-\frac1n))=\{-3,-2,-1, 0, \infty\}$} \\
\hline\hline
$\beta\in E(N(-1+\frac1n,-1-\frac1n))$ & $N(-1+\frac1n,-1-\frac1n))(\beta)$ \\ \hline\hline
$\beta=\infty$ & $S^3$ \\ \hline
$\beta=-3$ & $\lens{4n^2+3}{2n^2+n+2}$ \\ \hline
$\beta=-2$ & $\seiftre{S^2}32{1+n}n{1-n}n$ \\ \hline
$\beta=-1$ & $\seiftre{S^2}21{1+2n}{-n}{1-2n}n$ \\ \hline
$\beta=0$ & $\seifdue Dn{1+n}n{n-1}\bigu01{-1}{-1} \seifdue D2131$\\ \hline
%\end{tabular}
%\begin{tabular}{ |c|c| }
\hline
\multicolumn{2}{ |c| }{$n\geq 4$, 
\qquad $E(N(-1+\frac1n,-1-\frac1{n-2}))=\{-3,-2,-1, 0,\infty\}$} \\
\hline\hline
$\beta\in E(N(-1+\frac1n,-1-\frac1{n-2}))$ & $N(-1+\frac1n,-1-\frac1{n-2}))(\beta)$ \\ \hline\hline
$\beta=\infty$ & $S^3$ \\ \hline
$\beta=-3$ & $\lens{4n^2+8n-1}{2n^2-3n}$\\ \hline
$\beta=-2$ & $\seiftre{S^2}{1+n}n{3-n}{n-2}32$\\ \hline
$\beta=-1$ & $\seiftre{S^2}21{1+2n}{-n}{5-2n}{n-2}$ \\ \hline
$\beta=0$ & $\seifdue Dn{1+n}{2-n}{2-n}{3-3n}\bigu01{-1}{-1} \seifdue D2131$\\ \hline
\end{tabular}
\end{center}
\caption{The sets of exceptional slopes and fillings of all knot exteriors obtained by surgery 
on the minimally twisted 5-chain link realising $\Delta(T^H,T)=3$ or $\Delta(T^H,Z)=2$.\label{te}}
\end{table}

In Section \ref{T_5CL}, we show that if $\left(M_5(\mathcal
F),\alpha,\beta,\gamma\right)$ or $\left(M_4(\mathcal
F),\alpha,\beta,\gamma\right)$ is in $(S^H,T^H,T)$ with
$\Delta(\beta,\gamma)=3$, then $\mathcal F$ factors through $M_3$. 
If $M_3(\mathcal F)$ is hyperbolic then, by
\cite[Corollay~A.6]{Magic}, we know that either $e(M_3(\mathcal F))>5$, and then it appears in \cite[Tables~A.2--A.9]{Magic}, or $e(M_3(\mathcal F))=5$; Sections
\ref{3cl_t_enum}--\ref{t_3cl_5excep} investigate the exceptional
triples arising from $M_3$ in these two cases.
Note that \cite{Magic} 
classifies the exceptional filling instructions and fillings on $N$, the exterior of 
the mirror image of 3CL. Of course $N$ and $M_3$ are homeomorphic but, for the instructions, the slopes differ by a sign change; namely, $M_3(\alpha_1,\alpha_2,\alpha_3) = N(-\alpha_1,-\alpha_2,-\alpha_3)$.
For the sake of clarity, as we work with the Tables in \cite{Magic}, we will use the filling instructions on 
$N$.
Finally, Section \ref{t_conclusion} concludes the proof by comparing the
different families thus obtained. 

\subsection{Triples from $M_5$ and $M_4$}\label{T_5CL}
A complete enumeration of $E(M_5(\mathcal F))$, for $\mathcal F$ not
factoring through $M_4$, is given in
\cite[Theorem~4]{excep_slopes}. If $E(M_5(\mathcal
F))=\{0,1,\infty\}$, then all slopes are at distance 1. Moreover, a
careful inspection of \cite[Tables~6\,--11]{excep_slopes} shows that
only Table 6 has exceptional slopes at distance 3, but then
\cite[Table~14]{excep_slopes} shows that none of these examples
contains an $S^H$--slope. Any exceptional triple $(M_5(\mathcal F),\alpha,\beta,\gamma)\in(S^H,T^H,T)$ 
with $\Delta(\beta,\gamma)=3$ has $\mathcal F$ factoring through $M_4$. 

Similarly, \cite[Theorem~5]{excep_slopes} gives a complete enumeration
of $E(M_4(\mathcal F))$ for $\mathcal F$ not factoring through $M_3$.
Again, if $E(M_5(\mathcal
F))=\{0,1,2,\infty\}$ then all exceptional slopes are at distance at most 2,
and \cite[Tables\, 21--22]{excep_slopes} shows that, otherwise, there
is no example containing simultaneously $S^H$, $T^H$ and $T$
slopes. Any triple $(M_4(\mathcal F),\alpha,\beta,\gamma)\in(S^H,T^H,T)$ 
with $\Delta(\beta,\gamma)=3$ must have $\mathcal F$ factoring through $M_3$.

\subsection{Exceptional triples from $M_3(\mathcal F)$ with $e(M_3(\mathcal F))>5$}\label{3cl_t_enum}
We recall that, for the sake of clarity, we use here filling instructions on $N$, and that they actually differ in sign from the filling instructions on $M_3$. 

Any filling instruction $\mathcal F$ on $N$ consisting of two slopes and such that $e(N(\mathcal F))>5$ can be found in \cite[Tables~A.2-A.9]{Magic}. The tables A.2, A.3, A.4 and A.9 each contain a finite list of $N(\mathcal F)$. The remaining tables consist of four infinite families.
% ; Table A.5 considers $N(1,\frac rs)$ with $\frac pq$ exceptional, Table A.6 
% considers $N(-\frac32,\frac rs)$ with $\frac pq$ exceptional, Table A.7 considers $N(-\frac 52,\frac rs)$ 
% with $\frac pq$ exceptional, and finally Table A.8 considers $N(-\frac12,\frac rs)$ with $\frac pq$ exceptional.
We proceed to examine each of these tables in our quest for examples. 

\subsubsection{Examples arising from \cite[Tables A.2--A.4 and A.9]{Magic}}
The only hyperbolic knots, {\it i.e.} $N(\mathcal F)$ with an $S^H$--filling, listed are $N(1,2)$---also known as the Figure-8 knot---in Table A.2 and $N(-4,-\frac13)$---the 
$(-2,3,7)$ pretzel knot---in Table A.4. The former has no lens space filling while the latter gives 
a unique $(S^H,T^H,\Tor)$ triple with $\Delta(T^H,\Tor)=3$. So, from Tables A.2--A.4 and A.9 the only example we get 
is \fbox{$\left(N(-4,-\tfrac{1}{3}), \infty,0,-3\right)\in (S^H,T^H,\Tor).$}

\subsubsection{Examples arising from \cite[Table A.5]{Magic}}

This table enumerates $N(\mathcal F)=N(1,\frac rs)$ cases with exceptional 
$\frac pq\in E(N(\mathcal F))=\{-3,-2,-1,0,1,\infty\}$. By a direct
inspection, we see that a $S^H$--filling can arise in only two ways:
either  $\frac pq=\infty$ and $\frac rs=-5+\frac 1n$ with $n=0$ --- but
then $\frac rs=\infty$ so $\mathcal F$ contains $\infty$, which is
discarded by Lemma \ref{N_factor} --- or $\frac pq=\infty$ with
$r-s=\pm1$. In the latter case, up to a simultaneous change of sign
for $r$ and $s$, we can even assume w.l.o.g. that $r=s+1$. Moreover,
if $N(1,\frac rs)(\alpha)$ is toroidal then $\alpha$ is either $-3$ or $1$. We study both cases separately. 

\begin{description}
\item[If $-3$ is a toroidal slope on $N(1,\frac rs)$] Then $\frac
  pq=0$ is the only case satisfying $N(1,\frac rs)(\frac pq)\in T^H$ 
and $\Delta(\frac pq,-3)=3$. But moreover, $\frac rs$ would be equal
to $-5+\frac{1}{n}$ from \cite[Table A.5]{Magic} which is not compatible with the relation $r=s+1$,
otherwise we would obtain $6=\frac 1n-\frac
1s\in[-2,2]\cup\{\infty\}$. 
\item[If $1$ is a toroidal slope on $N(1,\frac rs)$)] Then $\frac
  pq=-2$ is the only case satisfying $N(1,\frac rs)$ hyperbolic, 
  $N(1,\frac rs)(\frac pq)\in T^H$ 
and $\Delta(\frac pq,1)=3$. 
But moreover, $\frac rs$ would be equal
to $-2+\frac{1}{n}$ which is not compatible with the relation $r=s+1$,
otherwise we would obtain $3=\frac 1n-\frac
1s\in[-2,2]\cup\{\infty\}$. 
\end{description} 

\subsubsection{Examples arising from \cite[Table A.6]{Magic}}

This table enumerates $N(\mathcal F)=N(-\frac32,\frac rs)$ cases with
exceptional $\frac pq\in E(N(\mathcal
F))=\{-3,-\frac52,-2,-1,0,\infty\}$. By a direct inspection, we see
that the possible $S^H$--slopes are $-3$, $-2$, $-1$ and $\infty$. Examining 
each possible case individually we find:
\begin{description}
\item[If $\frac{p}{q}=\infty$ is a $S^H$--slope] Then $\Delta(\frac pq,\alpha)\leq2$ for all $\alpha\in E(N(\mathcal F))$.  
\item[If $\frac{p}{q}=-3$ is a $S^H$--slope] Then
  $\frac{r}{s}=-1+\frac1n$ and $6n+7=\pm1$. The only possibility is
  $n=-1$ but then $\frac
  rs=-2$ and $N(\mathcal F)$ is non-hyperbolic by Lemma~\ref{N_factor}.
\item[If $\frac{p}{q}=-2$ is a $S^H$--slope] Then $4r+11s=\pm1$, and
  up to a simultaneous change of sign for $r$ and $s$, we may even
  assume that $4r+11s=1$, or equivalently that 
$\frac{r}{s}=\frac{1}{4s}-\frac{11}{4}$. The distance 3 pairs of slopes from $E(N(-\frac32,\frac rs))$ 
are $\{-3,0\}$ and $\{-\frac 52,-1\}$. In the first case, $-3$ must be the lens space surgery and $\frac rs$ 
is forced to be $-1+\frac1n$; then $-1+\frac1n=\frac rs=\frac{1}{4s}-\frac{11}{4}$ from where we arrive to the contradiction $\frac74=\frac1{4s}-\frac1n\leq\frac54$.
In the second case, $-1$ must be the lens space surgery and $\frac rs$ is 
forced to be $-3+\frac1n$; then $-3+\frac1n=\frac
rs=\frac{1}{4s}-\frac{11}{4}$, that is $4s-n=ns$. According to Lemma \ref{lem:PlentyEq} we have then $(n,s)\in\big\{(0,0),(3,3),(5,-5),(8,-2),(6,-3),(2,1)\big\}$.
\begin{description}
\item[Case $(n,s)=(0,0)$] Then $\frac{r}{s}=-3+\frac{1}{n}=\infty$ and the space is
  non-hyperbolic by Lemma~\ref{N_factor}.
\item[Case $(n,s)=(3,3)$] Then
  $\frac{r}{s}=-3+\frac{1}{n}=-\frac{8}{3}$ and this is excluded from Table A.6.
\item[Case $(n,s)=(5,-5)$] We obtain then \fbox{$\left(N(-\frac{3}{2},-\frac{14}{5}), -2,-1,-\frac{5}{2}\right)$} which
is indeed a $(S^H,T^H,\Tor)$ triple with $\Delta(T^H,\Tor)=3$.
\item[Case $(n,s)=(8,-2)$] then
  $r=\frac{1}{4}-\frac{11s}{4}=\frac{23}{4}\notin\Z$, which is a contradiction.
\item[Case $(n,s)=(6,-3)$] then
  $r=\frac{34}{4}\notin\Z$, which is a contradiction.
\item[Case $(n,s)=(2,1)$] then $r=-\frac{10}{4}\notin\Z$, which is a contradiction.
    \end{description}
\item[If $\frac{p}{q}=-1$ is a $S^H$--slope] Then
  $\frac{r}{s}=-3+\frac1n$ and $6n+1=\pm1$. The only possibility is
  $n=0$ but then $\frac
  rs=\infty$ and $N(\mathcal F)$ is non-hyperbolic by
  Lemma~\ref{N_factor}. 
\end{description}

\subsubsection{Examples arising from \cite[Table A.7]{Magic}}

This table enumerates $N(\mathcal F)=N(-\frac52,\frac rs)$ cases with exceptional
$\frac pq\in E(N(\mathcal F))=\{-3,-2,-\frac32,-1,0,\infty\}$. By a direct inspection, we see
that the possible $T^H$--slopes are $-2$, $-1$ and $\infty$. 
But none of these slopes are at distance 3 from any other slope in $E(N(\mathcal F))$.

\subsubsection{Examples arising from \cite[Table A.8]{Magic}}

This table enumerates $N(\mathcal F)=N(-\frac12,\frac rs)$ cases with 
exceptional $\frac pq\in E(N(\mathcal F))=\{-4,-3,-2,-1,0,\infty\}$. By a direct inspection, we see
that the possible $S^H$--slopes are $-3$, $-1$ and $\infty$.  
However, if $\frac pq\in\{-3,-1\}$ 
corresponds to an $S^H$--filling, then $\frac rs$ is either $-1+\frac
1n$ or $-3+\frac 1n$ with $n=0$, that is $\frac rs=\infty$, which
makes $N(\mathcal F)$ non-hyperbolic by Lemma~\ref{N_factor}. We can
hence assume that the $S^H$--slope is $\infty$ 
and, up to a simultabeous change of sign for $r$ and $s$, that $r+2s=1$, that
is $\frac{r}{s}=-2+\frac{1}{s}$.
Now, the only pairs of slopes at distance 3 in $E(N(\mathcal F))$ are
$(-4,-1)$ and $(-3,0)$, and Table A.8 
tells us that neither $-4$ or $0$ can correspond to a lens space
filling. On the other hand, if $\frac pq\in\{-3,-1\}$ 
corresponds to an $T^H$--filling, then $\frac rs$ is either $-1+\frac
1n$ or $-3+\frac 1n$. But
  since $\frac rs=-2+\frac1s$, it follows that $\frac rs$ is either
  $-\frac52$ or $-\frac 32$,
  which are both excluded from Table A.8.

\subsection{Exceptional triples arising from $N(\mathcal F)$ with $e(N(\mathcal F))=5$}\label{t_3cl_5excep}

The same arguments presented at the beginning of
Section~\ref{lens_lens3CL_subsec} reduce the study of  the cases
coming from \cite[Theorem~1.3 and Tables~2--4]{Magic} to just Table~2
and Theorem~1.3; namely to the hyperbolic 
$N(\frac rs,\frac tu)$ with $E(N(\frac rs,\frac tu))=\{-3,-2,-1,0,\infty\}$. Such 
$N(\frac rs,\frac tu)(\frac pq)$ can be toroidal only when $\frac{p}{q}=-3$ or 0.

\subsubsection{Case $\frac{p}{q}=-3$ is the $\Tor$--filling} In this case, \cite[Table~2]{Magic} 
gives us the conditions that $\frac rs,\frac tu\neq-1-\frac1n$. We also require the lens space slope 
to be at distance 3 from the toroidal slope so $\frac{p}{q}=0$ should
be the lens space slope and this implies that $\{\frac rs, \frac
tu\}=\{n, -4-n+\frac1m\}$. Up to symmetry, we may hence assume that $\frac{r}{s}=n$ and $\frac{t}{u}=-4-n+\frac{1}{m}$. The possible $S^H$--slopes 
are now $-1$, $-2$ and $\infty$. 
\begin{description}
\item[Case $\frac{p}{q}=-1$ is the $S^H$--slope] Then either
  $\frac rs=-3+\frac{1}{n'}=n$ or $\frac
  tu=-3+\frac{1}{n'}=-4-n+\frac{1}{m}$.\\
  If $\frac rs=-3+\frac{1}{n'}=n$, then $\frac{1}{n'}=3+n\in\Z$ so
    $3+n=\pm1$ and $n\in\{-4,-2\}$.  For $N(\mathcal F)$ to be
  hyperbolic, $n$ is necessarily $-4$ because of Lemma~\ref{N_factor}, so
  $\frac{t}{u}=\frac{1}{m}$. However, from \cite[Table~2]{Magic} we
  get then $N(\mathcal F)(-1)=N(-4,\frac1m)(-1)=L\big(-3-7m,\star\big)\neq
  S^3$.\\
If $\frac tu=-3+\frac{1}{n'}=-4-n+\frac{1}{m}$, then
$1+n=\frac{1}{m}-\frac{1}{n'}\in[-2,2]$ so $n\in
\{-3,-2,-1,0,1\}$. Because of Lemma~\ref{N_factor}, $N(\mathcal F)$ is
hyperbolic only when $n=1$. But then $n'=-1$ and
  $\frac{t}{u}=-4$ so, again from \cite[Table~2]{Magic}, $N(\mathcal F)(-1)=N(1,-4)(-1)= L(-10,\star)\neq S^3$.

\item[Case $\frac{p}{q}=-2$ is the $S^H$--slope] Then either
  $\frac rs=-2+\frac{1}{n'}=n$ or $\frac
  tu=-2+\frac{1}{n'}=-4-n+\frac{1}{m}$.\\
  If $\frac rs=-2+\frac{1}{n'}=n$, then $\frac{1}{n'}=2+n\in\Z$ so $n=-3$
  or $n=-1$. In both cases $N(\mathcal F)$ is non-hyperbolic by Lemma~\ref{N_factor}.\\
If $\frac tu=-2+\frac{1}{n'}=-4-n+\frac{1}{m}$, then $2+n=\frac{1}{m}-\frac{1}{n'}\in[-2,2]$, so
  $n\in\{-4,-3,-2,-1,0\}$. Because of Lemma~\ref{N_factor}, $N(\mathcal F)$ is
hyperbolic only when $n=-4$. But then $n'=1$ and
  $\frac tu=-2+\frac{1}{n'}=-1$, which makes $N(\mathcal F)$ non-hyperbolic by Lemma \ref{N_factor}.
\item[Case $\frac{p}{q}=\infty$ is the $S^H$--slope] Then from \cite[Theorem~1.3]{Magic} we know
 \[
N(\mathcal F)(\infty)=N\left(n,\frac{1-m(n+4)}{m}\right)(\infty)=L\Big(\big(1-m(n+4)\big)n-m,\star\Big).\] 
  For $N(\mathcal F)(\infty)$ to be $S^3$, it is hence required that
  $\big(1-m(n+4)\big)n=m\pm1$. By
  Lemma~\ref{lem:QuadEq}
  it follows then that $n\in\{-5,-4,-3,-2,-1,0,1\}$. 
  For $N(\mathcal F)=N(n,\frac{1-m(n+4)}{m})$ 
  to be hyperbolic, $n$ cannot be in $\{-3,-2,-1,0\}$ because of Lemma~\ref{N_factor}, so we are left
  with the cases $(n,m)\in\big\{(-5,-1),(-4,-5),(-4,-3),(1,0)\big\}$.
  The cases $(n,m)=(-5,-1), (1,0)$ yield again a non-hyperbolic $N(\mathcal F)$, while $(n,m)=(-4,-5), (-4,-3)$ 
  give the $(S^H,T^H, T)$ triples 
  \fbox{$\left(N(-4,-\frac{1}{5}), \infty,0,-3\right)$}
  and \fbox{$\left(N(-4,-\frac{1}{3}), \infty,0,-3\right)$}
  with $\Delta(T^H,\Tor)=3$.
\end{description}

\subsubsection{Case $\frac{p}{q}=0$ is a $\Tor$--filling}

To have a $T^H$--slope at distance 3 from the toroidal slope, we need that
$N(\frac rs,\frac tu)(-3)\in T^H$.
According to \cite[Table~2]{Magic}, it follows that either $\frac{r}{s}=-2$,
but then $N(\mathcal F)$ is non hyperbolic because of Lemma~\ref{N_factor}, or $\frac{r}{s}=-1+\frac{1}{n}$ and $\frac{t}{u}=-1+\frac{1}{m}$.
The $S^H$--slope is then one of $-1$, $-2$ or $\infty$.

\begin{description}
\item[Case $\frac{p}{q}=-2$ is the $S^H$--slope] Then, up to symmetry,
  $\frac rs=-2+\frac{1}{n'}=-1+\frac{1}{n}$ so $n'=2$. Since
  $\frac{t}{u}=\frac{1-m}{m}$, \cite[Table~2]{Magic} tells us that $N(\mathcal F)(-2)=L\big(4+7m,\star\big)\neq S^3$.
\item[Case $\frac{p}{q}=-1$ is the $S^H$--slope] Then, up to symmetry,
  $\frac rs=-3+\frac{1}{n'}=-1+\frac{1}{n}$ so
  $\frac{r}{s}=-2$ which makes $N(\mathcal F)$ non-hyperbolic by Lemma \ref{N_factor}.
\item[Case $\frac{p}{q}=\infty$ is the $S^H$--slope] Then from \cite[Theorem~1.3]{Magic}, we know
  \[
N(\mathcal
F)(\infty)=N\left(\frac{1-n}{n},\frac{1-m}{m}\right)(\infty)=L(1-n-m,\star).
\]
For $N(\mathcal F)(\infty)$ to be $S^3$, it is hence required that
$1-n-m=\pm1$, or equivalently that $n+m\in\{0,2\}$. The first case
leads to \fbox{$B_n:=\left(N(-1+\frac{1}{n},-1-\frac{1}{n}), \infty,-3,0\right)$}
  for $n\in\Z\setminus\{0,\pm1\}$
  and the second to
  \fbox{$C_n:=\left(N(-1+\frac{1}{n},-1-\frac{1}{n-2}), \infty,-3,0\right)$}
  for $n\in\Z\setminus\{0,\pm1,2,3\}$.
 Both families are indeed
$(S^H,T^H,\Tor)$ triples with $\Delta(T^H,\Tor)=3$.
\end{description}

\subsection{Conclusion}\label{t_conclusion}
Along Sections \ref{T_5CL}\,-- \ref{t_3cl_5excep}, we have proved that
the only $(S^H,T^H,T)$ triples obtained by surgery on the 5CL and that realize $\Delta(T^H,T)=3$ are:
  \begin{itemize}
  \item $B_n=
    N\left(-1+\frac{1}{n},-1-\frac{1}{n}\right)(\infty,-3,0)$ for
    $n\in\Z\setminus\{0,\pm1\}$;
  \item $C_n=
    N\left(-1+\frac{1}{n},-1-\frac{1}{n-2}\right)(\infty,-3,0)$ for
    $n\in\Z\setminus\{0,\pm1,2,3\}$;
  \item $N\left(-\frac{3}{2},-\frac{14}{5}\right)\left(-2,-1,-\frac{5}{2}\right)$;
  \item $N\left(-4,-\frac{1}{3}\right)(\infty,0,-3)$;
  \item $N\left(-4,-\frac{1}{5}\right)(\infty,0,-3)$.
  \end{itemize}
The last three isolated cases can be seen to be redundant using \cite[Theorem~1.5]{Magic}.
It follows indeed from the first equality
that $\left(N(-4,-\tfrac13), \infty,0,-3\right)\cong B_{-2}$
and
$\left( N(-4,-\tfrac{1}{5}),\infty,0,-3\right)\cong C_{-2}$, and
from the third equality that
$\left(N\left(-\tfrac{3}{2},-\tfrac{14}{5}\right),-2,-1,-\tfrac{5}{2}\right)\cong
C_{-2}$. 
This completes the proof that every $(S^H,T^H,T)$ triple with $\Delta(T^H,T)=3$ is 
equivalent to some $B_n$ or $C_n$.
But besides, $B_{-n}\cong B_{n}$ and $C_{-n}\cong C_{n+2}$ for all $n$ because of Lemma \ref{3CL_sym}.

We now show that these two families are 
distinct. This is done by comparing their exceptional fillings. 
Using \cite[Theorem~1.3]{Magic} and \cite[Table~2]{Magic} we can
indeed write down the 
exceptional slopes and fillings of both
$N\big(-1+\frac{1}{n},-1-\frac{1}{n}\big)$ and
$N\big(-1+\frac{1}{n},-1-\frac{1}{n-2}\big)$; the result is shown in
Table \ref{te}. We note that both $B_n$ and $C_n$ have a unique lens
space filling, namely $B_n(-3)=\lens{4n^2+3}{2n^2+n+2}$ and $C_{n}(-3)=\lens{4n^2+8n-1}{2n^2-3n}$. 
If $B_n(-3)=C_k(-3)$ for some $n,k\in\Z$, then the order of their
fundamental groups should be equal. 
But it is well known that $\pi_1(\lens pq)$ is the cyclic group of order $p$ 
(see for example \cite[Exercise 9.B.5]{Rolfsen}); so it would follow that 
$3+4n^2=4k^2+8k-1\Leftrightarrow 4(n-k)(k+n)=2(4k-1)$, and this would imply that 
$2\mid 4k-1$, which is a contradiction. Hence, $\{B_n\}\cap\{C_n\}=\emptyset$ and the proof of Theorem~\ref{lens_toroidal_prop} is complete.

\begin{remark}
  $B_{2}$ is the exterior of the $(-2,3,7)$ pretzel knot. In
  this case, $e(B_{2})=7$ and the exceptional slopes and fillings
  can be found in \cite[Table~A.2]{Magic}. In the second family,
  $E(C_{-2})=\{-3,-\frac52, -2,-1,0,\infty\}$; $C_{-2}(\alpha)$ is
  found in Table \ref{te} for
  $\alpha\in E(C_{-2})\backslash \{-\frac52\}$, and
  $C_{-2}(-\frac52)=\seifdue D2131\bigu110{-1} \seifdue D2154$.
\end{remark}

%\textcolor{red}{Does $B_{-n}\cong B_n$ and $C_{-n}\cong C_{n+2}$ ?}

\section{$(S^H,T^H,Z)$ triples}\label{lens_Z_sec}

In this section, we enumerate all $(S^H,T^H,Z)$ triples
obtained by surgery on the 5CL and realizing the maximal distance. It
shall turn out that all such triples are obtain by surgery on the 3CL. 

\begin{theorem}\label{lens_Z_prop}
  \begin{itemize}
  \item[]
  \item If
    $\left(M_5\big(\frac{p}q,\frac{r}s,\frac{u}v,
      \frac{x}y\big),\alpha, \beta, \gamma\right)\in (S^H,T^H,Z)$
    then $\Delta(\beta,\gamma)\leq2$.
  \item If
    $\left(M_5\big(\frac{p}q,\frac{r}s,\frac{u}v,
      \frac{x}y\big),\alpha, \beta, \gamma\right)\in (S^H,T^H,Z)$
    with $\Delta(\beta,\gamma)=2$ then it is equivalent to either 
$B'_n:=\left(N(-1+\tfrac{1}{n},-1-\tfrac{1}{n}), \infty,-3,-1\right)$ for some integer 
$n\geq 2$, or $C'_n :=\left(N(-1+\tfrac{1}{n},-1-\tfrac{1}{n-2}), \infty,-3,-1\right)$
for some integer $n\geq 4$.
  \end{itemize}
\end{theorem}

\begin{remark}
Note that the knot exteriors in Theorem~\ref{lens_Z_prop} are the same as the knot exteriors in Theorem~\ref{lens_toroidal_prop}. Therefore, we know that these examples are distinct and that the exceptional 
slopes and fillings are given in Table~\ref{te}.
\end{remark}

The proof shall proceed in three steps. In Section~\ref{5cl_z}, we show that if 
$(M_5(\mathcal F),\alpha,\beta,\gamma)\in(S^H,T^H,Z)$ and $\Delta(\beta,\gamma)\geq2$ then $\mathcal F$ factors through 
$M_4$. Then, in Section~\ref{4cl_z}, we show that if $(M_4(\mathcal F),\alpha,\beta,\gamma)\in(S^H,T^H,Z)$ 
and $\Delta(\beta,\gamma)\geq2$ then $\mathcal F$ factors through $M_3$. Finally, in Section~\ref{3cl_z_sec} we show that if 
$(N(\mathcal F),\alpha,\beta,\gamma)\in(S^H,T^H,Z)$ then
$\Delta(\beta,\gamma)\leq2$ and all 
triples with $\Delta(\beta,\gamma)=2$ are then enumerated.

\subsection{$(S^H,T^H,Z)$ triples from $M_5$}\label{5cl_z}

If $(M_5(\mathcal F),\alpha,\beta,\gamma)\in(S^H,T^H,Z)$ and $\mathcal F$ does not factor through $M_4$ then, from \cite[Theorem~4]{excep_slopes}, 
we have that either $E(M_5(\mathcal F))=\{0,1,\infty\}$, or $\mathcal F$ is equivalent to one of the surgery instructions in
\cite[Tables 14 -- 20]{excep_slopes}. But clearly, if $E(M_5(\mathcal
F))=\{0,1,\infty\}$ then all exceptional slopes are at distance 1. Any
$(S^H,T^H,Z)$ triple realising $\Delta(T^H,Z)\geq2$ should hence be
found in \cite[Tables 14 -- 20]{excep_slopes}. In these tables, the
only simultaneous occurence of $S^H$, $T^H$ and $Z$--slopes are in
Table 17, with $\mathcal F=(-2,\frac pq,3,\frac uv)$, and in Table 18,
with $\mathcal F=(-2,\frac pq,\frac rs,-2)$. In both cases, the
exceptional slopes are $0$, $\pm1$ and $\infty$, so as a consequence
$\Delta(\beta,\gamma)\leq 2$. Moreover, the only possibility for
$(Z,T^H)$--slopes to realise $\Delta(\beta,\gamma)= 2$ is that
$\{\beta,\gamma\}=\{\pm1\}$; the $S^H$--slope is then either 0 or
$\infty$. We proceed now with a case by case analysis.

\begin{description}
\item[Case $\mathcal F=(-2,\frac pq,3,\frac uv)$] By applying $\textrm{(\ref{linksymeq1})}^{-2}\circ\textrm{(\ref{symeq6})}$, we
  may assume that 0 corresponds to the $S^H$--slope. It follows then from
\cite[Table~17]{excep_slopes} that either $\frac pq=1+\frac1n$  and 
$|(3+2n)u-(7+6n)v|=1$, or $\frac uv=3+\frac1k$ and
$|(3+2k)p-(1+2k)q|=1$.
Moreover, from $-1$ being a $T^H$ or $Z$--slope, we also know that
either $|p|=1$ or $|u+v|=1$. These conditions shall be shown to be incompatible. 
\begin{description}
\item[Case $\frac pq=1+\frac1n$  and $|(3+2n)u-(7+6n)v|=1$]$\ $\\
 If $|p|=1$ then, up to reversing both $p$ and $q$, we may assume that
  $p=1$ so that $-1+\frac1{1-q}=n\in\Z$, that is $q\in\{0,2\}$. But
  then $\frac pq\in\{\frac 12,\infty\}$, which is discarded by Lemma
  \ref{5CL_factor}.
  
If $|u+v|=1$ then, up to reversing both $u$ and $v$, we may assume
that $u=1-v$. Subbing this into $|(3+2n)u-(7+6n)v|=1$ 
and solving for $v$ in terms of $n$, we obtain that $v$ is either
$\frac{2+n}{5+4n}$ or $\frac{1+n}{5+4n}$.
But $v\in\Z$, so in the first case, $n$ shall be $-1$ or $-2$, that is
$\frac uv\in\{0,\infty\}$; and in the second case, $n$ shall be $-1$,
that is $\frac uv=\infty$. All these cases are discarded by Lemma
  \ref{5CL_factor}.
\item[Case $\frac uv=3+\frac1k$ and $|(3+2k)p-(1+2k)q|=1$]$\ $\\
If $|u+v|=1$, then $1=|u+v|=|(3k+1)+k|$ meaning that $k=0$ 
and that $\frac uv=\infty$ which is excluded by Lemma \ref{5CL_factor}.

If $|p|=1$ then, up to reversing both $p$ and $q$, we may assume that
$p=1$ and subbing this in $|(3+2k)p-(1+2k)q|=1$ we obtain that $q$ is
either $1+\frac{1}{1+2k}$ or $1+\frac{3}{1+2k}$. But $q\in\Z$, so
$q\in\{0,\pm2,4\}$. If $q\in\{0,2\}$, then $\frac pq\in\{\frac
12,\infty\}$ and this is discarded by Lemma
\ref{5CL_factor}. If $q\in\{-2,4\}$, then $k\in\{-1,0\}$ and $\frac
uv\in\{2,\infty\}$ which is also discarded by Lemma
\ref{5CL_factor}.
\end{description}

\item[Case $\mathcal F=(-2,\frac pq,\frac rs,-2)$] From \cite[Table~18]{excep_slopes} we see that 
for $-1$ to be a type $Z$ or $T^H$--slope we need $|q|=1$ or $|s|=1$.
By applying $\textrm{(\ref{linksymeq1})}^{-1}\circ\textrm{(\ref{linksymeq2})}$, we may assume that $|q|=1$.
But from the same table, we also see that for $1$ to be a type $Z$ or
$T^H$--slope we need $|p|=1$ or $|r|=1$. Since $|q|=1$, the case $|p|=1$
is discarded by Lemma \ref{5CL_factor}. It follows hence that $|r|=1$.
But now, \cite[Table~18]{excep_slopes} also tells that the only 
possible $S^H$--slope is 0, and that it requires either $\frac pq=1+\frac1n$ or 
$\frac rs=1+\frac1n$. But, since $|q|=1$, the first condition implies
that $\frac pq\in\{0,2\}$ and, since $|r|=1$, the second condition
implies that $\frac rs=\frac12$; all these are discarded by Lemma
\ref{5CL_factor}.
\end{description}

\subsection{$(S^H,T^H,Z)$ triples from $M_4$}\label{4cl_z}

If $(M_4(\mathcal F),\alpha,\beta,\gamma)\in(S^H,T^H,Z)$ with
$\mathcal F$ not factoring through $M_3$, then according to
\cite[Theorem~5]{excep_slopes}, either
$E(M_4(\mathcal F))=\{0,1,2,\infty\}$ or $\mathcal F$ is equivalent to a filling instruction listed in \cite[Tables~
21\,-- 22]{excep_slopes}.
But in these tables, $S^H$ and $T^H$--slopes never occur
simultaneously. It follows that $E(M_4(\mathcal F))=\{0,1,2,\infty\}$.
In particular, $\Delta(\beta,\gamma)\leq2$ and if $\Delta(\beta,\gamma)=2$ then $\{\beta,\gamma\}=\{0,2\}$ and $\alpha\in\{1,\infty\}$. 
But one can observe that, on one hand,
\begin{eqnarray*}
M_4(\tfrac ab,\tfrac cd,\tfrac ef,\tfrac gh)
& \mathop{\cong}\limits_{\textrm{Lemma }\ref{lem:M5M4}}&
M_5(\tfrac ab,\tfrac{c-d}d,-1,\tfrac{e-f}f,\tfrac gh)
\
                                                         \mathop{\cong}\limits_{(\ref{symeq3})} \
M_5(\tfrac{g}{g-h}, \tfrac{b-a}b, -1,\tfrac{d}{c-d}, \tfrac{2f-e}f)\\
&\mathop{\cong}\limits_{\textrm{Lemma }\ref{lem:M5M4}}&
M_4(\tfrac g{g-h}, \tfrac{2b-a}{b}, \tfrac{c}{c-d}, \tfrac{2f-e}f)
\ \mathop{\cong}\limits_{\textrm{Lemma }\ref{4CL_symmetry}}\ 
M_4(\tfrac{2b-a}{b}, \tfrac{c}{c-d}, \tfrac{2f-e}f, \tfrac g{g-h}),
\end{eqnarray*}
and that, on the other hand,
\begin{eqnarray*}
M_4(\tfrac ab,\tfrac cd,\tfrac ef,\tfrac gh)
& \mathop{\cong}\limits_{\textrm{Lemma }\ref{lem:M5M4}}&
M_5(\tfrac ab,\tfrac{c-d}d,-1,\tfrac{e-f}f,\tfrac gh)
\ \mathop{\cong}\limits_{(\ref{linksymeq1})^2\circ (\ref{symeq1})} \
M_5(\tfrac{2d-c}{d}, \tfrac{f}{e-f}, -1, \tfrac{h-g}{h}, \tfrac{a}{a-b}) \\
&\mathop{\cong}\limits_{\textrm{Lemma }\ref{lem:M5M4}}&
M_4(\tfrac{2d-c}{d}, \tfrac{e}{e-f}, \tfrac{2h-g}{h}, \tfrac{a}{a-b})
\ \mathop{\cong}\limits_{\textrm{Lemma }\ref{4CL_symmetry}}\ 
M_4(\tfrac{a}{a-b}, \tfrac{2d-c}d, \tfrac e{e-f}, \tfrac{2h-g}{h}).
\end{eqnarray*}
Consequently, it follows then directly that $\left(M_4(\tfrac ab,\tfrac cd,\tfrac
  ef),1,0,2\right)\cong\left(M_4(\tfrac{e-2f}{e-f},\tfrac{c-2d}{c-d},\tfrac{2b-a}{b-a}),\infty,0,2\right)$
and that $\left(M_4\big(\tfrac{p}q,\tfrac{r}s,\tfrac{u}v\big), \infty,
  0,
  2\right)\cong\left(M_4\big(\tfrac{p}{p-q},\tfrac{2s-r}{s},\tfrac{u}{u-v}\big),\infty,
  2, 0\right)$. Up to these equivalencies, we can hence assume that $\infty$ is the $S^H$--slope, $0$
is the $T^H$--slope and $2$ the $Z$--slope. We set the filling instruction on $M_4$ to be $\mathcal F=(\frac
ab,\frac cd, \frac ef)$.

Since $M_4(\mathcal F)(\infty)=S^3$, we know
by (\ref{4CLinf}) and Lemma \ref{l:down}, that one of $|a|$, $|d|$ or $|e|$ is 1.

Since $M_4(\mathcal F)(0)$ is $T^H$, we know
by (\ref{4CL0}) and Lemma \ref{l:down}, that one of $b$, $f$ or $c-2d$ is in
$\{0,\pm1\}$. 
But if one of them is 0, then one of $\frac ab,\frac cd, \frac ef$ is
in $\{2, \infty\}$ and 
$M_4(\tfrac ab, \tfrac cd, \tfrac ef)$ is non-hyperbolic by Lemma
\ref{4CL_factor}. We conclude that one of $|b|$, $|f|$ or $|c-2d|$ is
1, that is either $\frac ab =n$, $\frac ef =n$ or $\frac cd =2+\frac
1k$. Using Lemma \ref{4CL_symmetry}, we may even assume that either $\frac ab =n$ or $\frac cd =2+\frac
1k$.

Since $M_4(\mathcal F)(2)\in Z$, we know
by (\ref{4CL2}) and Lemma \ref{l:down}, that one of $a-b$, $c$ or $e-f$ is in
$\{0,\pm1\}$. 
But if one of them is 0, then one of $\frac ab,\frac cd, \frac ef$ is
in $\{0,1\}$ and 
$M_4(\tfrac ab, \tfrac cd, \tfrac ef)$ is non-hyperbolic by Lemma
\ref{4CL_factor}. We conclude that one of $|a-b|$, $|c|$ or $|e-f|$ is
1, that is either $\frac ab =1+\frac1p$, $\frac ef =1+\frac1p$ or $\frac cd =\frac1p$.

Collecting the necessary conditions for $(\infty,0,2)$ to be a $(S^H,T^H, Z)$ triple found above, we see that 
at least one condition from each column in Table \ref{table_4CL_lensZ}
must be fulfilled.
\begin{table}
\begin{center}
\begin{tabular}{|c|c|c|}
\hline $\infty\leadsto S^H$ & $0\leadsto T^H$ & $2\leadsto Z$ \\ \hline \hline 
$a=\pm1$ & $\frac ab=n$ & $\frac ab=1+\frac1p$ \\ \hline
$d=\pm1$ & $\frac cd=2+\frac1k$& $\frac ef=1+\frac1p$ \\ \hline 
$e=\pm1$ &  & $\frac cd=\frac1p$ \\ \hline
\end{tabular}
\vspace{1mm}
\caption{Necessary conditions for $(\infty,0,2)$ to be a $(S^H,T^H, Z)$ triple. \label{table_4CL_lensZ} }
\end{center}
\end{table}

\begin{description}
\item[Case $\frac cd=2+\frac1k$]
Then $\frac cd=\frac{2k+1}k$ and identity (\ref{4CL0}) implies that
  \begin{gather*}
    M_4(\tfrac ab,\tfrac cd,\tfrac ef)(0)
    \mathop{\cong}\limits_{(\ref{4CL0})}
    \seifdue{D}f{-e}b{2b-a}\bigu0110\seifdue{D}211k\\
    \mathop{\cong}\limits_{(\ref{graph_eq1})}
    \seiftre{S^2}f{-e}b{2b-a}{2k+1}{-2}.
  \end{gather*}
But since $M_4(\mathcal F)(0)$ is a lens space, and
according to Lemma \ref{l:down}, it follows that either $b$, $f$ or
$2k+1$ is $\pm1$.

If $2k+1=\pm1$, then
$k\in\{-1,0\}$ and $\frac cd=2+\frac1k\in\{1,\infty\}$, which is
ruled out by Lemma \ref{4CL_factor}.

If $|b|$ or $|f|$ is 1, then up to Lemma
\ref{4CL_symmetry}, we may even assume that $|b|=1$. But
now, we can claim that $a,d\neq\pm1$, otherwise $\frac ab$ would be
$\pm1$, or $\frac cd=2+\frac 1k$ would be 1 or 3, and those are
discarded by Lemma
\ref{4CL_factor}. Looking at the first column of Table \ref{table_4CL_lensZ}, we
conclude hence that $e=\pm1$.
Looking now at the third column of Table \ref{table_4CL_lensZ}, we see
that either $\frac ab=1+\frac1p$, but then condition $b=\pm1$ implies that
$\frac ab\in\{0,2\}$; or $\frac ef=1+\frac 1p$, but then condition
$e=\pm1$ implies that $\frac ef=\frac12$, or $\frac
cd=\frac1p$, but then condition $\frac cd=2+\frac1k$ implies that $\frac
cd=1$. All those are discarded by Lemma
\ref{4CL_factor}.
\item[Case $\frac ab=n$] 
Then identity (\ref{4CL0}) implies that
  \begin{gather*}
    M_4(\tfrac ab,\tfrac cd,\tfrac ef)(0) \mathop{=}\limits_{(\ref{4CL0})}
    \seifdue{D}f{-e}1{2-n} \bigu0110 \seifdue{D}21{c-2d}d
    \mathop{=}\limits_{(\ref{graph_eq1})}
    \seiftre{S^2}21{c-2d}d{f(2-n)-e}{-f}.
  \end{gather*}
But since $M_4(\mathcal F)(0)$ is a lens space, and
according to Lemma \ref{l:down}, it follows that either $c-2d$ or
$e+f(n-2)$ is $\pm1$.

If $|c-2d|=1$, then $\frac cd=2+\frac1k$ and we
are left to the previous case.

If $e+f(n-2)=\pm1$, that is $\frac ef=2-n+\frac1k$, then $\frac ab\neq1+\frac1p$, otherwise $\frac ab=n$ would be
0 or 2 and this is discarded by Lemma \ref{4CL_factor}. 
But $\frac ef$ is also distinct from $1+\frac1p$. Indeed,
$1+\frac1p=2-n+\frac1k$ would imply that
$(n,k)\in\big\{\big(1,p),(0,2),(2,2)\}$ and then either $\frac ab=n=1$
or $\frac ef\in\{\frac12,\frac32\}$, both are discarded by Lemma
\ref{4CL_factor}.
Looking at the third column of Table \ref{table_4CL_lensZ}, we
conclude hence that $\frac cd=\frac 1p$.
Looking now at the first column of Table \ref{table_4CL_lensZ}, we see
that either $a=\pm1$, but then condition $\frac ab=n$ implies that
$\frac ab=\pm1$; or $d=\pm1$, but then condition
$\frac cd=\frac 1p$ implies that $\frac ef=\pm1$, or $e=\pm1$, but
then condition $\frac ef=2-n+\frac1k$ implies that $\frac
ab=n\in\{1,2,3\}$. All those are discarded by Lemma \ref{4CL_factor}.
\end{description}

This ends the proof that if 
$\left(M_5\big(\mathcal F\big),\alpha, \beta, \gamma\right)\in (S^H,T^H,Z)$ 
and $\Delta(\beta,\gamma)\geq2$ then $\mathcal F$ factors through $M_3$.

\subsection{$(S^H,T^H,Z)$ triples from $M_3$}\label{3cl_z_sec}

\begin{proposition}\label{prop:Dist>2}
If $\left(N\big(\frac rs, \frac tu),\alpha, \beta, \gamma\right)\in (S^H,T^H,Z)$ 
then $\Delta(\beta,\gamma)\leq2$.
\end{proposition}

\begin{proof}
If $N\big(\frac rs, \frac tu\big)$ is hyperbolic, then \cite[Corollary~A.6]{Magic} tells us that either $e\big(N(\frac rs, \frac tu)\big)=5$ or $N\big(\frac rs, \frac tu\big)$ is found in  \cite[Tables~A.2\,-- A.9]{Magic}. Moreover, if $e\big(N(\frac rs, \frac tu)\big)=5$ then, it is a consequence of \cite[Theorem~1.3]{Magic} and Lemma~\ref{N_factor} that $E\big(N(\frac rs, \frac tu)\big)=\{\infty,-3,-2,-1,0\}$. Since we just want to dismiss pairs of exceptional slopes at distance greater than two, we only have to consider the case $\{\beta,\gamma\}=\{0,-3\}$.

\begin{description}
\item[Case $\{\beta,\gamma\}=\{0,-3\}$] In
  this case, we can see in \cite[Table~2]{Magic} that if $N(\frac
  rs,\frac tu)$ is hyperbolic with $N(\frac rs,\frac tu)(-3)\in T^H$
  then either $\frac rs=-2$, but this is dismissed by Lemma~\ref{N_factor}, or $\frac rs=-1+\frac 1n$ and $\frac tu=-1+\frac 1m$. But then
  $N(\frac rs,\frac tu)(0)\in Z$ and \cite[Table~2]{Magic} tells us that
  one of $\frac rs=-1+\frac 1n$ or $\frac tu=-1+\frac1m$ is an
  integer, so that one of $\frac rs$ and $\frac tu$ is in $\{-2,0\}$,
  which is forbidden by Lemma
  \ref{N_factor}.

  On the other hand, if $N(\frac rs,\frac tu)(-3)\in Z$, then the same table tells
  us\footnote{A word of
    caution: as one can read in the arXiv
    preprint of this article, the relation between the entries on the
    $\tfrac rs$ column and the $\tfrac tu$ column in \cite[Table~2]{Magic} is more intricate
    than what a reader might appreciate in the published
    version, and the conditions $\frac rs=-1+\frac 1n$ and $\frac rs\neq-2$ are actually shared by lines 4 and 5.}\label{f:1} that, up to  Lemma \ref{3CL_sym}, $\frac rs=-1+\frac1n$. But then $N(\frac rs,\frac tu)(0)\in T^H$ and \cite[Table~2]{Magic} tells us that $\big\{\frac rs,\frac tu\big\}=\big\{k,-4-k+\frac 1m\big\}$. The case $\frac rs=-1+\frac 1n=k$ would imply $\frac rs\in\{-2,0\}$ and is hence dismissed by Lemma~\ref{N_factor}; so we can assume that $\frac
  rs=-1+\frac 1n=-4-k+\frac1m$ and $\frac tu=k$.  As $n=\pm1$ would make
  $N(\frac rs,\frac tu)$ non-hyperbolic because of
  Lemma~\ref{N_factor}, we have $3+k=\frac1m-\frac1n\in\{0,\pm1\}$, that is $k\in\{-4,-3,-2\}$. But since $\frac tu=k$, cases $k=-3$ and $k=-2$ are dismissed by Lemma~\ref{N_factor}. If $k=-4$ then $\frac tu=-4$ and $\frac rs=-1+\frac
  1n=\frac1m$; it follows that $\frac rs=-\frac12$, but $N(\frac
  rs,\tfrac tu)=N(-4,-\frac12)$ is non-hyperbolic, see
  \cite[Table~1]{Magic}.

  \item[Case $N\big(\frac rs, \frac tu\big)$ is found in {\cite[Tables A.2\, -- A.9]{Magic}} and $\{\beta,\gamma\}\neq\{0,-3\}$] It is
  immediately clear that the only $(S^H,T^H,Z)$ triple in Tables A.2--A.4 and Table A.9 is the triple obtained from the $(-2,3,7)$ pretzel
  knot, and in this case $\Delta(\beta,\gamma)=2$.

  If $\left(N(\frac rs,\tfrac
    tu),\alpha,\beta,\gamma\right)\in(S^H,T^H,Z)$ with
  $\Delta(\beta,\gamma)>2$ is found in \cite[Table~A.5]{Magic} then
  $E(N(\frac rs,\tfrac tu))=\{-3,-2,-1,0,1,\infty\}$ so that $1\in\{\beta,\gamma\}$. However, this table tells us that $N(\frac rs,\tfrac
  tu)(1)$ is never in $T^H\cup Z$.

  If $\left(N(\frac rs,\tfrac
    tu),\alpha,\beta,\gamma\right)\in(S^H,T^H,Z)$ with
  $\Delta(\beta,\gamma)>2$ is found in \cite[Table~A.6]{Magic} then
  $E(N(\frac rs,\tfrac tu))=\{-3,-\frac 52,-2,-1,0,\infty\}$ and $-\frac52\in\{\beta,\gamma\}$. This table also tells us that $N(\frac rs,\tfrac
  tu)(-\frac52)\notin T^H$ and that $N(\frac rs,\tfrac
  tu)(-\frac52)\in Z$ only when $\frac rs=-2+\frac1n$. Moreover, if
  $\Delta(\beta,-\frac52)>2$ then $\beta\in\{-1,0\}$; but 0 is not a $T^{H}$--slope and if $-1$ is,
  then $\frac
  rs=-2+\frac1n=-3+\frac1k$, that is $\frac rs=-\frac52$, which is
  actually excluded from this table.

  If $\left(N(\frac rs,\tfrac
    tu),\alpha,\beta,\gamma\right)\in(S^H,T^H,Z)$ with
  $\Delta(\beta,\gamma)>2$ is found in \cite[Table~A.7]{Magic} then
  $E(N(\frac rs,\tfrac tu))=\{-3,-2,-\frac32,-1,0,\infty\}$ and $-\frac32\in\{\beta,\gamma\}$. Moreover this table tells us $N(\frac rs,\tfrac
  tu)(-\frac32)\not\in T^H$ so that $\gamma=-\frac 32$; but if
  $\Delta(\beta,-\frac32)>2$ then $\beta\in\{-3,0\}$ and neither of them is a $T^H$--slope.

  Finally, if $\left(N(\frac rs,\tfrac
    tu),\alpha,\beta,\gamma\right)\in(S^H,T^H,Z)$ with
  $\Delta(\beta,\gamma)>2$ is found in \cite[Table~A.8]{Magic} then
  $E(N(\frac rs,\tfrac tu))=\{-4,-3,-2,-1,0,\infty\}$ and $-4\in\{\beta,\gamma\}$.
This table also tells us that $N(\frac rs,\tfrac
  tu)(-4)\not\in T^H$ and that $N(\frac rs,\tfrac tu)(-4)\in Z$ only
  when $\frac rs\in \mathbb{Z}$. Moreover, if $\Delta(\beta,-4)>2$ then
  $\beta\in\{-1,0\}$; but 0 is not a $T^H$--slope, and if $-1$ is, then $\frac rs=-3+\frac 1n$ and, since $\frac rs\in\Z$, $\frac rs\in\{-4,-2\}$, which are
  actually excluded from this table.
\end{description}

\end{proof} 

\begin{proposition}\label{lens_Z_prop2}
If $\left(N\big(\frac rs, \frac tu),\alpha, \beta, \gamma\right)\in (S^H,T^H,Z)$ 
and $\Delta(\beta,\gamma)=2$ then it is equivalent to either $B'_n:=\left(N(-1+\tfrac{1}{n},-1-\tfrac{1}{n}), \infty,-3,-1\right)$
with 
$n\in\Z\setminus\{0,\pm1\}$, or to $C'_n :=\left(N(-1+\tfrac{1}{n},-1-\tfrac{1}{n-2}), \infty,-3,-1\right)$
with $n\in\Z\setminus\{0,\pm1,2,3\}$.
\end{proposition}

\begin{proof}
By the same discussion that the one which begins the proof of Proposition \ref{prop:Dist>2}, we know  that either 
$E(N\big(\frac rs, \frac tu))=\{0,-1,-2,-3,\infty\}$ or $\{0,-1,-2,-3,\infty\}\varsubsetneq E(N\big(\frac rs, \frac tu))$ 
and, in the latter case, $N\big(\frac rs, \frac tu)$ and $E(N\big(\frac rs, \frac tu))$ are found in \cite[Tables~A.2\,-- A.9]{Magic}. 

\begin{description}
\item[Case $\{\alpha, \beta,
    \gamma\}\not\subset\{0,-1,-2,-3,\infty\}$] In this case, $N(\frac
  rs,\frac tu)$, $E(N(\frac rs,\frac tu))$ are found in
  \cite[Tables~A.2\,-- A.9]{Magic}.  It is immediately clear that the
  only $\left(N\big(\frac rs, \frac tu),\alpha, \beta,
    \gamma\right)\in (S^H,T^H,Z)$ in \cite[Tables~A.2\,-- A.4 and
  Table~A.9]{Magic} is the $(-2,3,7)$ pretzel knot exterior
  \fbox{$\left(N\big(-4, -\frac 13),\infty, 0, -2\right)$}.

  If  $\left(N\big(\frac rs, \frac tu),\alpha,
    \beta, \gamma\right)\in (S^H,T^H,Z)$ is found in
  \cite[Table~A.5]{Magic} then the exceptional set $E(N\big(\frac rs,
  \frac tu))$ is $\{-3,-2,-1,-0,1,\infty\}$ so that $1\in\{\alpha,\beta,\gamma\}$; but in this table $N(1,\frac
  rs)(1)\not\in T^H\cup S^H\cup Z$.

  If $\left(N\big(\frac rs, \frac
    tu\big),\alpha, \beta, \gamma\right)\in (S^H,T^H,Z)$ is found in
  \cite[Table~A.6]{Magic} then the exceptional set
  $E(N\big(\frac rs, \frac tu))$ is $\{-3,-\frac52,-2,-1,-0,\infty\}$ so that $-\frac 52\in\{\alpha,\beta,\gamma\}$. Moreover this table tells us that $N(\frac rs,
  -\frac 32)(-\frac 52)$ is in $S^H\cup T^H\cup Z$ only if $\frac
  rs=-2+\frac 1n$, in which case $N(-2+\frac{1}{n}, -\frac 32)(-\frac
  52)\in Z$. But we have $\Delta(\beta,-\frac52)= 2$ for $\beta\in
  E(N\big(\frac rs, \frac tu))$ only when $\beta=\infty$.
  We can also see from Table~A.6 that the only possible $S^H$--slopes on
  hyperbolic $N(\frac rs,-\frac32)$ are $\infty$ which is already the
  $T^H$--slope in our case, $-3$ but then $\frac rs=-2$ and this is discarded by Lemma~\ref{N_factor}, $-2$, and $-1$  but then $\frac rs=\infty$ and this is discarded by Lemma~\ref{N_factor}. We are hence left with $\alpha=-2$ and $|4r+11s|=1$. But $\frac rs=\frac{1-2n}{n}$, so $|4r+11s|=1$ if and only if
  $n=-1$. It would follow that $\frac rs=-3$ and this is excluded by Lemma
  \ref{N_factor}.

  If $\left(N\big(\frac rs, \frac
    tu\big),\alpha, \beta, \gamma\right)\in (S^H,T^H,Z)$ is found in
  \cite[Table~A.7]{Magic} then the exceptional set
  $E(N\big(\frac rs, \frac tu))$ is $\{-3,-2,-\frac32,-1,0,\infty\}$ so that $-\frac 32\in\{\alpha,\beta,\gamma\}$. This
  table also tells us that $N(\frac rs, -\frac 52)(-\frac 32)$ is in
  $S^H\cup T^H\cup Z$ only if $\frac rs=-2+\frac 1n$, in which case
  $N(-2+\frac{1}{n}, -\frac 52)(-\frac 32)\in Z$. Moreover, the only possible
  $S^H$--slopes found in this table are $\infty$, $-2$ and
  $-1$. More precisely, one can read that
  \begin{itemize}
  \item $N\left(-2+\tfrac{1}{n}, -\tfrac
      32\right)(\infty)=\lens{5-8n}{\star}$ which is $S^H$ if
    and only if $|5-8n|=1$, but this has no integer solution;
  \item $N\left(-2+\tfrac{1}{n}, -\tfrac 32\right)(-2)=\lens{8-3n}{\star}$
    which is $S^H$ if and only if $|8-3n|=1$, that is when $n=3$, but then $\frac rs=-\frac 53$ and we get 
    the $(-2,3,7)$ pretzel knot which is
    excluded from \cite[Table~A.7]{Magic};
  \item $N\left(-2+\tfrac{1}{n}, -\tfrac
      32\right)(-1)
    =\lens{3+5n}{\star}$
    which is $S^H$ if and only if $|3+5n|=1$, but this has no integer
    solution.
  \end{itemize}

  If $\left(N\big(\frac rs, \frac
    tu\big),\alpha, \beta, \gamma\right)\in (S^H,T^H,Z)$ is found in 
  \cite[Table~A.8]{Magic} then the exceptional set
  $E(N\big(\frac rs, \frac tu))$ is $\{-4,-3,-2,-1,0,\infty\}$ so that $-4\in\{\alpha,\beta,\gamma\}$. This table also
  tells us that $N(\frac rs, -\frac 12)(-4)$ is in $S^H\cup T^H\cup Z$
  only if $\frac rs=n$, in which case $N(n, -\frac 12)(-4)\in
  Z$. Moreover, the only possible
  $S^H$--slope are $\infty$, $-3$, $-2$ and $-1$, but in the last three cases $\frac rs=\infty$ and this is discarded by Lemma~\ref{N_factor}. Finally, $N\left(n, -\tfrac 12\right)(\infty)=\lens{n+2}{\star}$
  is $S^H$ if and only if $n=-3$, which makes $N\left(n, -\frac
    12\right)$ non-hyperbolic by Lemma~\ref{N_factor}.

  \item[Case $\{\alpha, \beta,
    \gamma\}\subset\{0,-1,-2,-3,\infty\}$] All examples can be
  constructed from \cite[Table~2]{Magic}. However, as noted in the footnote on page \pageref{f:1}, we warn the reader that Table 2, as given in the published version, misses some separating lines. We recommend hence to look at the arXiv version.

  We first show that we may assume that $\infty$
    corresponds to the $S^H$--slope. Indeed, we know by \cite{CGLS} that the distance between the $S^H$
  and the $T^H$--slope is 1, and we are looking to  $(S^H,T^H,Z)$ triple realizing $\Delta(T^H,Z)=2$. So, if $\infty$ is not an
  $S^H$--slope, then the triple $(\alpha,\beta,\gamma)$ belongs to $\big\{(-3,-2,0),(-2,-3,-1),(-2,-1,-3),(-1,-2,0),(-1,0,-2),(0,-1,-3)\big\}$.

  \begin{description}
  \item[Case $(\alpha,\beta,\gamma)=(-3,-2,0)$] Then for $-3$ to be a $S^H$--slope, we need that either $\frac rs=-2$, but this is discarded by Lemma~\ref{N_factor}, or $\frac{r}{s}=-1+\frac{1}{n}$ and $\frac tu=-1+\frac{1}{m}$. But since $0$ is a $Z$--slope, one of $\frac rs$ or $\frac tu$ has to be in $\Z$ and hence equal to $-2$ or 0. This is again forbidden by Lemma~\ref{N_factor}.

\item[Cases $(\alpha,\beta,\gamma)=(-2,-3,-1)$ and $(\alpha,\beta,\gamma)=(-2,-1,-3)$] 
Using Lemma \ref{3CL_sym} and Identity~(1.3) in
\cite[Proposition~1.5]{Magic}, we obtain that any such example is going to be 
equivalent to $\left( N\big(-\tfrac32,\tfrac{2t+5u}{t+2u}\big),\infty,-\tfrac{2\beta+5}{\beta+2},-\tfrac{2\gamma+5}{\gamma+2}\right)$, where $\infty$ is the $S^H$--slope. This identifies the triple
$(\alpha,\beta,\gamma)=(-2,-3,-1)$ with $(\alpha,\beta,\gamma)=(\infty,-1,-3)$ and the triple $(\alpha,\beta,\gamma)=(-2,-1,-3)$ with 
$(\alpha,\beta,\gamma)=(\infty,-3,-1)$ which are considered later.

%    \item[Case $(\alpha,\beta,\gamma)=(-2,-3,-1)$] Then for $-3$ to be a $T^H$--slope, we need that either $\frac rs=-2$, but this is discarded by Lemma~\ref{N_factor}, or $\frac{r}{s}=-1+\frac{1}{n}$ and $\frac tu=-1+\frac{1}{m}$. But for $-2$ to be a $S^H$--slope, one of $\frac rs$ or $\frac tu$ has to be of the form $-2+\frac 1k$ and hence equal to $-\frac 32$. Using Lemma \ref{3CL_sym} and Identity~(1.3) in \cite[Proposition~1.5]{Magic}, we obtain that any such example is equivalent to $\left( N\big(-\tfrac32,\tfrac{2t+5u}{t+2u}\big),\infty,-\tfrac{2\beta+5}{\beta+2},-\tfrac{2\gamma+5}{\gamma+2}\right)$ where $\infty$ is the $S^H$--slope.
%
%    \item[Case $(\alpha,\beta,\gamma)=(-2,-1,-3)$] Then, for $-3$ to be a $Z$--slope, we need that $\frac rs=-1+\frac1n$. Moreover, for $-1$ to be a $T^H$--slope, we need that either $\frac rs$ or $\frac tu$ is of the form $-3+\frac 1k$ and this cannot be $\frac rs$ otherwise we would have $\frac rs=-2$, which is discarded by Lemma~\ref{N_factor}. Finally, for $-2$ to be a $S^H$--slope, we need that either $\frac rs$ or $\frac tu$ is of the form $-2+\frac 1l$.
%
%If $\frac rs=-2+\frac 1l=-1+\frac 1n$, then $\frac rs=-\frac 32$ and, using Lemma \ref{3CL_sym} and Identity~(1.3) in \cite[Proposition~1.5]{Magic}, we see that it is equivalent to an example where $\infty$ is the $S^H$--slope.
%
%If $\frac tu=-2+\frac 1l=-3+\frac 1k$, then $\frac tu=-\frac 52$, but $N\big(-1+\tfrac1n, -\tfrac52\big)(-2)=\lens{3(-2)(1-n+2n)-2(1-n)-n}{\star}=\lens{-8-5n}{\star}\neq S^3$.
 
    \item[Case $(\alpha,\beta,\gamma)=(-1,-2,0)$] Then, for $-2$ to be a $T^H$--slope, we need that $\frac rs=-2+\frac1n$. Moreover, for $-1$ to be a $S^H$--slope, we need that either $\frac rs$ or $\frac tu$ is of the form $-3+\frac 1k$.

      If $\frac rs=-3+\frac 1k=-2+\frac 1n$, then $\frac rs=-\frac52$ so for 0 to a $Z$--slope, $\frac tu$ must be an integer $l$. In this case, $N\big(-\tfrac52,l\big)(-1)=\lens{3l+11}{\star}$ is $S^3$ if and only if $l=-4$; and indeed \fbox{$\left(N\big(-4,-\tfrac52\big),-1,-2,0\right)$} is a $\in(S^H,T^H,Z)$ triple.

       If $\frac tu=-3+\frac 1k$, then for 0 to a $Z$--slope, $\frac rs=-2+\frac1n$ or $\frac tu=-3+\frac 1k$ must be an integer. But the values $-3$, $-2$ and $-1$ are all discarded by Lemma~\ref{N_factor}, so the only remaining case is $\frac tu=-4$.  In this case, $N\big(-4,-2+\frac1n\big)(-1)=\lens{-3-n}{\star}$ is $S^3$ if and only if $n\in\{-2,-4\}$; and indeed $\left(N\big(-\tfrac52,
      -4\big),-1,-2,0\right)$ (already listed) and \fbox{$\left(N\big(-4,-\tfrac94\big),-1,-2,0\right)$} are $(S^H,T^H,Z)$ triples.

     \item[Case $(\alpha,\beta,\gamma)=(-1,0,-2)$]  Then, for $-1$ to be a $S^H$--slope, we need that $\frac rs=-3+\frac1n$. Moreover, for $0$ to be a $T^H$--slope, we need that $\big\{\frac rs,\frac tu\big\}=\big\{k,-4-k+\frac1m\big\}$.

       If $\frac rs=k=-3+\frac1n$, then either $\frac rs=-2$, but this is discarded by  Lemma~\ref{N_factor}, or $\frac rs=-4$ and then $\frac tu=\frac1m$. In this case $N\big(-4,\tfrac1m\big)(-1)=\lens{-3-7m}{\star}\neq S^3$.

       If $\frac rs=-4-k+\frac1m=-3+\frac1n$ then $k=-1+\frac1m-\frac1n$ so that $\frac tu=k\in\{-3,-2,-1,0,1\}$. But the values $-3$, $-2$, $-1$ and 0 are all discarded by Lemma~\ref{N_factor}, so the only remaining case is $\frac tu=1$, implying that $n=-1$ so that $\frac rs=-4$. In this case, $N\big(-4,1\big)(-1)=\lens{-10}{\star}\neq S^3$.

    \item[Case $(\alpha,\beta,\gamma)=(0,-1,-3)$] Then, for $0$ to be a $S^H$--slope, we need that $\big\{\frac rs,\frac tu\big\}=\big\{n,-4-n+\frac1m\big\}$ and moreover that $m=0$. It follows that either $\frac rs$ or $\frac tu$ is $\infty$, which is discarded by Lemma~\ref{N_factor}.
  \end{description}

We can now assume that $\infty$ is the $S^H$--slope, that is $\alpha=\infty$ and
  \begin{equation}\label{c1}
    |rt-su|=1.
  \end{equation}
 Moreover, $\{\beta,\gamma\}\subset\{-3,-2,-1,0\}$, so for $\Delta(\beta,\gamma)=2$ to hold, the only possibilities are $(\beta,\gamma)\in\big\{(-3,-1),(-1,3),(-2,0),(0,-2)\big\}$.

 \begin{description}
  \item[Case $(\alpha,\beta,\gamma)=(\infty,-3,-1)$] Then for $-3$ to be a $T^H$--slope, we need that either $\frac rs=-2$, but this is discarded by Lemma~\ref{N_factor}, or $\frac{r}{s}=-1+\frac{1}{n}$ and $\frac tu=-1+\frac{1}{m}$. Condition (\ref{c1}) becomes then $(1-n)(1-m)-nm=\pm1$, that is $ m=1\pm1-n$. This leads to \fbox{$\left(N\big(-1+\frac1n,-1-\frac1n\big),\infty,-3,-1\right)$} and \fbox{$\left(N\big(-1+\frac1n,-1-\frac1{n-2}\big),\infty,-3,-1\right)$} which are indeed families of $(S^H,T^H,Z)$ triples.

    \item[Case $(\alpha,\beta,\gamma)=(\infty,-1,-3)$] Then for $-1$ to be a $T^H$--slope, we need that $\frac rs=-3+\frac1n$. Moreover, for  $-3$ to be a $Z$--slope, we need that either $\frac rs$ or $\frac tu$ is of the form $-1+\frac 1k$ and this cannot be $\frac rs$ otherwise we would have $\frac rs=-2$, which is discarded by Lemma~\ref{N_factor}. We have then $\big(\frac rs,\frac tu\big)=\big(-3+\frac1n,-1+\frac1m\big)$ and condition (\ref{c1}) becomes $(1-3n)(1-m)-nm=\pm1$, that is $m=\frac{3n}{2n-1}$ or $\frac{3n-2}{2n-1}$. This implies that $n\in\{-1,0,1,2\}$ in the former case and that $n\in\{0,1\}$
   in the latter. Cases $n\in\{0,1\}$ make $\frac rs\in\{-2,\infty\}$ and are excluded by Lemma \ref{N_factor}. If $n=-1$
   then $m=1$ and $\frac tu=0$ which is also
   excluded by Lemma \ref{N_factor}. If $n=2$ then
   $m=2$, $\frac rs=-\frac52$ and $\frac tu=-\frac12$; and indeed \fbox{$\left(N\big(-\tfrac52,-\tfrac12\big),\infty,-1,-3\right)$} is a $(S^H,T^H,Z)$ triple.

 \item[Case $(\alpha,\beta,\gamma)=(\infty,-2,0)$] Then for $-2$ to be a $T^H$--slope, we need that $\frac rs=-2+\frac1n$. Moreover, for  $0$ to be a $Z$--slope, we need that either $\frac rs$ or $\frac tu$ is an integer $k$ and this cannot be $\frac rs$ otherwise we would have $\frac rs\in\{-3,-1\}$, which is discarded by Lemma~\ref{N_factor}. We have then $\big(\frac rs,\frac tu\big)=\big(-2+\frac1n,k\big)$ and condition (\ref{c1}) becomes $k(1-2n)-n=\pm1$, that is $k=\frac{n+1}{1-2n}$ or $\frac{n-1}{1-2n}$. This implies that $(n,k)\in\big\{(-1,0),(0,-1),(0,1),(1,-2),(1,0),(2,-1)\big\}$, and in all cases we have either $n\in\{-1,0,1\}$ or $k=-1$, that is either $\frac rs\in\{-3,-1,\infty\}$ or $\frac tu=-1$. All are discarded by Lemma \ref{N_factor}.

    \item[Case $(\alpha,\beta,\gamma)=(\infty,0,-2)$] Then for $0$ to be a $T^H$--slope, we need that $\frac rs=n$ and $\frac tu=-4-n+\frac1m$. Condition (\ref{c1}) becomes $n(1-4m-nm)=m\pm1$ whose solutions are given Lemma \ref{lem:QuadEq}. First, we can exclude all solution with $\frac rs=n\in\{-3,-2,-1,0\}$ which are discarded by Lemma \ref{N_factor}. We also exclude $(n,k)=(-5,-1)$ and $(n,k)=(1,0)$ which give $\frac tu\in\{0,\infty\}$, also discarded by Lemma \ref{N_factor}. We are then left with $\left(N\big(-4,-\frac13\big),\infty,0,-2\right)$ (already listed) and \fbox{$\left(N\big(-4,-\frac15\big),\infty,0,-2\right)$} which are indeed $(S^H,T^H,Z)$ triples.
 \end{description}
\end{description}
  In the above analysis, we have proved that every $\left(N\big(\frac rs, \frac
    tu),\alpha, \beta, \gamma\right)$ which is a $(S^H,T^H,Z)$ triple with
  $\Delta(\beta,\gamma)=2$ is equivalent to one of
  \begin{enumerate}[(i)]
  \item\label{cas1}
    $B'_n:=
    \left(N(-1+\frac1n,-1-\frac1n),\infty,-3,-1\right)$
    for $n\in\Z\setminus\{0,\pm1\}$;
%and since $B'_{n}=B'_{-n}$ we can restrict to $n\geq 2$;
 \item\label{cas2}
    $C'_n:=
    \left(N(-1+\frac1n,-1-\frac1{n-2}),\infty,-3,-1\right)$
    for $n\in\Z\setminus\{0,\pm1,2,3\}$;
%, and since $C'_{-n}=C'_{n+2}$ we can restrict to $n\geq 4$;
  \end{enumerate}
\begin{multicols}{2}
  \begin{enumerate}[(i)]
    \setcounter{enumi}{2}
 \item\label{cas3} $\left(N(-4,-\frac13),\infty,0,-2)\right)$;
 \item\label{cas4} $\left(N(-4,-\frac15),\infty,0,-2)\right)$;
 \item\label{cas5} $\left(N(-4,-\frac52),-1,-2,0)\right)$;
 \item\label{cas6} $\left(N(-4,-\frac94),-1,-2,0)\right)$;
 \item\label{cas7} $\left(N(-\frac52,-\frac12),\infty,-1,-3)\right)$.
  \end{enumerate}
\end{multicols}

But now, using Identitiy (1.1) in \cite[Proposition 1.5]{Magic}, we obtain that cases (\ref{cas3}), (\ref{cas4}), (\ref{cas5}) and (\ref{cas6}) are respectivelu equivalent to $B'_{-2}$, $C'_{-2}$, $\left(N(-\frac32,-\frac83),-2,-1,-3)\right)$ and $\left(N(-\frac32,-\frac{14}5),-2,-1,-3)\right)$; and that, using Identitiy (1.3) in \cite[Proposition 1.5]{Magic}, that the latter two are equivalent to, respectively, $B'_2$ and $C'_2$. Finally, using Lemma \ref{3CL_sym} and  Identitiy (1.4) in \cite[Proposition 1.5]{Magic}, we obtain that case (\ref{cas7}) is equivalent to $B'_2$. Moreover, $B'_{-n}\cong B'_n$ and $C'_{-n}\cong C_{n+2}$ for every $n\geq 2$ because of Lemma \ref{3CL_sym}, and we already observed in Section \ref{t_conclusion} that the two families are
distinct. This completes the proof.
\end{proof}

\appendix
\section{Facts used liberally throughout this article}

The classification in this article comes from a careful consideration of the tables found in 
\cite{Magic} and \cite{excep_slopes}. Often, cases considered in the enumeration are identified 
and/or discounted using technical results, most of which are found in \cite{Magic} and 
\cite{excep_slopes}. To keep this article as self-contained as possible we list the technical 
lemmata that are used in this article. 

\subsection{Identities between graph manifolds}

The following lemma consists of a list of identities between graph manifolds which are found in both 
\cite{excep_slopes} and \cite{Magic}. Details can be found in \cite{FomMat}.

\begin{lemma}\label{graph_id_lemma}
The following identities on graph manifolds hold:

\begin{align}\label{graph_eq1}
\seifdue D1bcd\bigu 0110 \seifdue Defgh &= \seifdue Defgh\bigu 0110 \seifdue D1bcd \\&=
\seiftre{S^2}efgh{d+bc}{-c}\nonumber 
\end{align}
\begin{gather} 
\seiftre{S^2}abcd01= \lens ab\# \lens cd \label{graph_eq2} \\
\seiftre{S^2}abcd1e = \lens{a(d+ce)+bc}{\star}\label{graph_eq3}
\end{gather}

\end{lemma}

The following obvious lemma is used throughout the article. 
\begin{lemma}\label{l:down}
  \begin{itemize}
  \item[]
  \item If $\seifdue Dabcd\bigu 0110\seifdue Defgh$ is a Seifert space, a lens space or $S^{3}$, then one of
    $|a|,|c|,|e|$ or $|g|$ is less than or equal to 1.
  \item If $\seiftre {S^{2}}abcdef$ is a lens space or $S^{3}$, then one of
    $|a|$, $|c|$ or $|e|$ is equal to 1.
  \end{itemize}
\end{lemma}

\subsection{Concerning surgery instruction on 5CL} 

\begin{lemma}[{\cite[Lemma~2.2]{excep_slopes}}]\label{sym} The action of \emph{Aut}$(M_5)$ on surgery instructions on \emph{5CL} is generated by 
\emph{(\ref{linksymeq1})--(\ref{symeq11})}.
Moreover, for $\ref{symeq1}\leq n \leq \ref{symeq11}$ each \emph{($n$)} 
corresponds to the action of a distinct element of 
\emph{Aut}$(M_5)/G$ where $G$ is the subgroup generated by the elements \emph{(\ref{linksymeq1})--(\ref{linksymeq2})} 
corresponding to the generators of the link symmetry group of \emph{5CL}.
\begin{equation}\label{linksymeq1}
\left( \alpha_1, \alpha_2, \alpha_3, \alpha_4, \alpha_5 \right) \, \, \, \longmapsto \, \, \, 
\left( \alpha_5, \alpha_1, \alpha_2, \alpha_3, \alpha_4 \right) 
\end{equation}
\begin{equation}\label{linksymeq2}
\left( \alpha_1, \alpha_2, \alpha_3, \alpha_4, \alpha_5 \right) \, \, \, \longmapsto \, \, \, \left( \alpha_5, \alpha_4, \alpha_3, \alpha_2, \alpha_1 \right)\end{equation}
\begin{equation}\label{symeq1}
\big( \tfrac{a}{b}, \tfrac{c}{d}, \tfrac{e}{f}, \tfrac{g}{h}, \tfrac{i}{j} \big) \, \, \, \longmapsto \, \, \, \big( \tfrac{f}{e}, \tfrac{j-i}{j}, \tfrac{a}{a-b}, \tfrac{d-c}{d}, \tfrac{h}{g} \big) 
\end{equation}
\begin{equation}\label{symeq2}
\big( \tfrac{a}{b}, \tfrac{c}{d}, \tfrac{e}{f}, \tfrac{g}{h}, \tfrac{i}{j} \big) \, \, \, \longmapsto \, \, \, \big( \tfrac{b}{b-a}, \tfrac{i-j}{i}, \tfrac{e-f}{e}, \tfrac{d}{d-c}, \tfrac{g}{h} \big) 
\end{equation}
\begin{equation}\label{symeq3}
\big( \tfrac{a}{b}, \tfrac{c}{d}, \tfrac{e}{f}, \tfrac{g}{h}, \tfrac{i}{j} \big) \, \, \, \longmapsto \, \, \, \big( \tfrac{i}{i-j}, \tfrac{b-a}{b}, \tfrac{f}{e}, \tfrac{d}{c}, \tfrac{h-g}{h} \big) 
\end{equation}
\begin{equation}\label{symeq4}
\big( \tfrac{a}{b}, \tfrac{c}{d}, \tfrac{e}{f}, \tfrac{g}{h}, \tfrac{i}{j} \big) \, \, \, \longmapsto \, \, \, \big( \tfrac{j}{j-i}, \tfrac{e}{f}, \tfrac{b}{b-a}, \tfrac{c-d}{c}, \tfrac{g-h}{g} \big)  
\end{equation}
\begin{equation}\label{symeq5}
\big( \tfrac{a}{b}, \tfrac{c}{d}, \tfrac{e}{f}, \tfrac{g}{h}, \tfrac{i}{j} \big) \, \, \, \longmapsto \, \, \, \big( \tfrac{a}{a-b}, \tfrac{e}{e-f}, \tfrac{i}{i-j}, \tfrac{c}{c-d}, \tfrac{g}{g-h} \big) 
\end{equation}
\begin{equation}\label{symeq6}
\big( \tfrac{a}{b}, \tfrac{c}{d}, \tfrac{e}{f}, \tfrac{g}{h}, \tfrac{i}{j} \big) \, \, \, \longmapsto \, \, \, \big( \tfrac{h}{g}, \tfrac{j}{i}, \tfrac{f-e}{f}, \tfrac{c}{c-d}, \tfrac{b-a}{b} \big)
\end{equation}
\begin{equation}\label{symeq7}
\big( \tfrac{a}{b}, \tfrac{c}{d}, \tfrac{e}{f}, \tfrac{g}{h}, \tfrac{i}{j} \big) \, \, \, \longmapsto \, \, \, \big( \tfrac{h}{h-g}, \tfrac{a}{b}, \tfrac{f}{f-e}, \tfrac{c-d}{c}, \tfrac{i-j}{i} \big) 
\end{equation}
\begin{equation}\label{symeq8}
\big( \tfrac{a}{b}, \tfrac{c}{d}, \tfrac{e}{f}, \tfrac{g}{h}, \tfrac{i}{j} \big) \, \, \, \longmapsto \, \, \, \big( \tfrac{g}{g-h}, \tfrac{f-e}{f}, \tfrac{b}{a}, \tfrac{d}{c}, \tfrac{j-i}{j} \big)
\end{equation}
\begin{equation}\label{symeq9}
\big( \tfrac{a}{b}, \tfrac{c}{d}, \tfrac{e}{f}, \tfrac{g}{h}, \tfrac{i}{j} \big) \, \, \, \longmapsto \, \, \, \big( \tfrac{g-h}{g}, \tfrac{f}{f-e}, \tfrac{i}{j}, \tfrac{d}{d-c}, \tfrac{a-b}{a} \big) 
\end{equation}
\begin{equation}\label{symeq10}
\big( \tfrac{a}{b}, \tfrac{c}{d}, \tfrac{e}{f}, \tfrac{g}{h}, \tfrac{i}{j} \big) \, \, \, \longmapsto \, \, \, \big( \tfrac{h-g}{h}, \tfrac{b}{a}, \tfrac{j}{i}, \tfrac{d-c}{d}, \tfrac{e}{e-f} \big)
\end{equation}
\begin{equation}\label{symeq11}
\big( \tfrac{a}{b}, \tfrac{c}{d}, \tfrac{e}{f}, \tfrac{g}{h}, \tfrac{i}{j} \big) \, \, \, \longmapsto \, \, \, \big( \tfrac{a-b}{a}, \tfrac{e-f}{e}, \tfrac{h}{h-g}, \tfrac{c}{d}, \tfrac{j}{j-i} \big).
\end{equation}
\end{lemma}

\begin{lemma}[{\cite[Theorem~4 and Eq. (70)]{excep_slopes}}]\label{5CL_factor} The following statements hold: 	
\begin{itemize}
\item If $0,1,\infty\in\{\frac ab,\frac cd,\frac ef,\frac gh,\frac ij\}$ then 
$M_5(\frac ab,\frac cd,\frac ef,\frac gh,\frac ij)$ is non-hyperbolic.
\item If $-1,\frac12,2\in\{\frac ab,\frac cd,\frac ef,\frac gh,\frac ij\}$ then 
$(\frac ab,\frac cd,\frac ef,\frac gh,\frac ij)$ factors through $M_4$. 
\end{itemize}
\end{lemma}

As highlighted in \cite{MPR}:

\begin{lemma}\label{lem:M5M4}
The following identity holds:
\begin{equation}\label{5CL-4CL}
M_5(\tfrac ab,\tfrac cd, -1, \tfrac ef,\tfrac gh)=M_4(\tfrac ab,\tfrac{c+d}d,\tfrac{e+f}f,\tfrac gh).
\end{equation}
\end{lemma}

\subsection{Concerning surgery instructions on 4CL} 

From \cite{excep_slopes} we have the following identities:
\begin{gather}
M_4(\tfrac ab, \tfrac cd, \tfrac ef)(\infty)= \seiftre{S^2}abd{-c}ef, \label{4CLinf} \\
M_4(\tfrac ab, \tfrac cd, \tfrac ef)(0)= \seifdue{D}f{-e}b{2b-a} \bigu0110 \seifdue{D}21{c-2d}d, \label{4CL0}\\
M_4(\tfrac ab, \tfrac cd, \tfrac ef)(1)= \seiftre{S^2}{a-2b}b{c-d}c{e-2f}f, \label{4CL1} \\
M_4(\tfrac ab, \tfrac cd, \tfrac ef)(2)= \seifdue{D}{a-b}b{e-f}f \bigu0110 \seifdue{D}cd2{-1}. \label{4CL2}
\end{gather}

\begin{lemma}[{\cite[Theorem~5 and Eq. (69)]{excep_slopes}}]\label{4CL_factor} The following statements hold:
\begin{itemize}
\item If $(\frac ab, \frac cd, \frac ef, \frac gh)$ is an instruction on \emph{4CL} and one of the slopes is 
in $\{0,1,2,\infty\}$ then $M_4(\frac ab, \frac cd, \frac ef, \frac gh)$ is non-hyperbolic. 
\item If $(\frac ab, \frac cd, \frac ef, \frac gh)$ is an instruction on \emph{4CL} and one of the slopes is 
in $\{-1,\frac12, \frac32,3\}$ then $(\frac ab, \frac cd, \frac ef, \frac gh)$ factors through $M_3$. In particular, $M_4(\frac ab,-1,\frac cd,\frac ef)=M_3(\frac ab+1,\frac cd+1,\frac ef)$
\end{itemize}
\end{lemma}

%\begin{proof}
%All the actions refer to \cite[Theorem~1.0.7]{thesis}.
%\[
%\dessin{2cm}{D3_1_1}\xrightarrow{surgery}\dessin{2cm}{D3_1_2}
%\]
%  \[
%\dessin{2cm}{D3_5_1}\xrightarrow{surgery}\dessin{2cm}{D3_5_2}\xrightarrow[\frac{c}{d}=\frac{3}{2},\frac{g}{h}=-1]{\textrm{action }1.4}\dessin{2cm}{D3_5_3}\xrightarrow{surgery}\dessin{2cm}{D3_5_4}
%\]
%  \[
%\dessin{2cm}{D3_a_1}\xrightarrow{surgery}\dessin{2cm}{D3_a_2}\xrightarrow[\frac{a}{b}=-1,\frac{c}{d}=-\frac{1}{2}]{\textrm{action }1.12}\dessin{2cm}{D3_a_3}\xrightarrow{surgery}\dessin{2cm}{D3_a_4}
%\]
%  \[
%\dessin{2cm}{D3_3_1}\xrightarrow{surgery}\dessin{2cm}{D3_3_2}\xrightarrow[\frac{e}{f}=-1,\frac{g}{h}=2]{\textrm{action }1.5}\dessin{2cm}{D3_3_3}\xrightarrow{surgery}\dessin{2cm}{D3_3_4}
%\]
%\end{proof}

%The action induced from the link symmetry group of 4CL tells us: 

\begin{lemma}\label{4CL_symmetry} For a filling instruction $(\alpha_1,\alpha_2,\alpha_3,\alpha_4)$ on 
$M_4$ we have
$M_4(\alpha_1,\alpha_2,\alpha_3,\alpha_4)=M_4(\alpha_{\sigma(1)},\alpha_{\sigma(2)},\alpha_{\sigma(3)},\alpha_{\sigma(4)})$
for every $\sigma\in D_4$.
% \[
% M_4(\tfrac ab, \tfrac cd, \tfrac ef, \tfrac gh) = M_4(\tfrac ef, \tfrac cd, \tfrac ab, \tfrac gh)
% \qquad \text{ and }\qquad 
% M_4(\tfrac ab, \tfrac cd, \tfrac ef, \tfrac gh) = M_4(\tfrac gh, \tfrac ab, \tfrac cd, \tfrac ef).
% \]
\end{lemma}

\subsection{Concerning surgery instructions on 3CL} 

\begin{lemma}\label{3CL_sym}
If $\sigma\in S_3$ and $(\alpha_1,\alpha_2,\alpha_3)$ is a filling instruction on $N$ 
then 
\[
N(\alpha_1,\alpha_2,\alpha_3)=N(\alpha_{\sigma(1)},\alpha_{\sigma(2)},\alpha_{\sigma(3)}).
\]
\end{lemma}

\begin{lemma}\label{M3vsN}
For all filling instructions it holds $M_3(\frac ab,\frac cd,\frac ef)=N(-\frac ab,-\frac cd,-\frac ef)$.
\end{lemma}

\begin{lemma}[{\cite[Theorem~1.2]{Magic}}]\label{N_factor}
If $(\frac ab, \frac cd)$ is an instruction on $N$ and one of the slopes is 
$\{0,-1,-2,-3,\infty\}$ then $N(\frac ab, \frac cd)$ is non-hyperbolic. 
\end{lemma}

% The following Lemma is contained in the statement of \cite[Corollary~A.6]{Magic}.

% \begin{lemma}\label{N_factor} The following statements hold:
% \begin{itemize}
% \item If $(\frac ab, \frac cd, \frac ef)$ is an instruction on $N$ and one of the slopes is 
% $\{0,-1,-2,-3,\infty\}$ then $N(\frac ab, \frac cd, \frac ef)$ is non-hyperbolic. 
% \item If $(\frac ab, \frac cd, \frac ef)$ is an instruction on $N$ and one of the slopes is 
% $\{-4, -\frac52, -\frac32, -1, -\frac 12,1\}$ then $N(\frac ab, \frac cd, \frac ef)$ is found in 
% Tables A.2--A.8 of \cite{Magic}.
% \end{itemize}
% \end{lemma}

\subsection{Concerning surgery instructions on M4CL}\label{m4cl_ref}

\cite[Proposition~2.1]{excep_slopes} gives us a complete enumeration of the 
Dehn fillings on $F$, the exterior of the minimally twisted 4 chain link. 
We have: 

\begin{lemma}\label{filling_F}
For slopes $\frac ab$, $\frac cd$, $\frac ef$, $\frac gh$ on \emph{M4CL} the following 
identity holds:
\begin{equation}\label{fillingF}
F(\tfrac ab, \tfrac ef, \tfrac cd, \tfrac gh)= \seifdue Dabcd\bigu 0110\seifdue Defgh
\end{equation} 
\end{lemma}

\begin{lemma}\label{D4-F}
  For a filling instruction $(\alpha_1,\alpha_2,\alpha_3,\alpha_4)$ on 
$F$ we have
$F(\alpha_1,\alpha_2,\alpha_3,\alpha_4)=F(\alpha_{\sigma(1)},\alpha_{\sigma(2)},\alpha_{\sigma(3)},\alpha_{\sigma(4)})$
for every $\sigma\in D_4$.
\end{lemma}

In fact, ``most" exceptional fillings of $M_5$ are obtained by filling $F$
(c.f.\ \cite[Proposition~3.1]{excep_slopes}). 

\begin{lemma}\label{5CL-F}
The following identities hold:
\begin{gather}
M_5(\tfrac ab, \tfrac cd, \tfrac ef, \tfrac gh)(\infty)= F(-\tfrac ab, \tfrac fe, \tfrac dc, -\tfrac gh) \label{5CLinf} \\
M_5(\tfrac ab, \tfrac cd, \tfrac ef, \tfrac gh)(1)= F(\tfrac{a-b}b, \tfrac cd, \tfrac ef, \tfrac{g-h}h) \label{5CL1} \\
M_5(\tfrac ab, \tfrac cd, \tfrac ef, \tfrac gh)(0)= F(\tfrac{b}{b-a},\tfrac{c-d}c,-\tfrac hg, \tfrac{e-f}f)
\label{5CL0}
\end{gather}
\end{lemma}

Consequently, 

\begin{lemma}\label{F-5CL}
The following identities hold:
\begin{gather}
F(\tfrac ab, \tfrac cd, \tfrac ef, \tfrac gh)=M_5(-\tfrac ab, \tfrac fe, \tfrac dc, -\tfrac gh)(\infty) 
\label{Finf} \\
F(\tfrac ab, \tfrac cd, \tfrac ef, \tfrac gh)=M_5(\tfrac{a+b}b, \tfrac cd, \tfrac ef, \tfrac{g+h}h)(1)
\label{F1} \\
F(\tfrac ab, \tfrac cd, \tfrac ef, \tfrac gh)=M_5(\tfrac {a-b}a, \tfrac d{d-c}, -\tfrac hg, \tfrac{f+e}f)(0)
\label{F0}
\end{gather}
\end{lemma}

\subsection{Some elementary diophantine equations}

\begin{lemma}
  \label{lem:PlentyEq}
  For $(n,s)\in\Z^2$, we have
  \begin{itemize}
  \item $s-n=ns\ \Longrightarrow\ (n,s)\in\big\{(0,0),(2,-2)\big\}$~;
  \item $2s-n=ns\ \Longrightarrow\ ((n,s)\in\big\{(0,0),(1,1),(3,-3),(4,-2)\big\}$~;
  \item $4s-n=ns\ \Longrightarrow\ ((n,s)\in\big\{(0,0),(3,3),(5,-5),(8,-2),(6,-3),(2,1)\big\}$~;
  \item $s-n=3ns\ \Longrightarrow\ (n,s)\in\big\{(0,0)\big\}$~;
  \item $2s-n=3ns\ \Longrightarrow\ ((n,s)\in\big\{(0,0),(1,-1)\big\}$~;
  \item $4s-n=3ns\ \Longrightarrow\ ((n,s)\in\big\{(0,0),(1,1),(2,-1)\big\}$~;
  \item $8s-n=3ns\ \Longrightarrow\ (n,s)\in\big\{(0,0),(3,-3),(2,1),(4,-1)\big\}$~;
  \item $5s-n=3ns\ \Longrightarrow\ ((n,s)\in\big\{(0,0),(2,-2)\big\}$~;
  \item $s-n=-5ns\ \Longrightarrow\ (n,s)\in\big\{(0,0)\big\}$~;
  \item $2s-n=-5ns\ \Longrightarrow\ (n,s)\in\big\{(0,0)\big\}$~;
  \item $4s-n=-5ns\ \Longrightarrow\ (n,s)\in\big\{(0,0),(-1,1)\big\}$~;
  \item $8s-n=-5ns\ \Longrightarrow\ (n,s)\in\big\{(0,0),(-2,1)\big\}$~:
  \item $3s-n=-5ns\ \Longrightarrow\ (n,s)\in\big\{(0,0)\big\}$.
  \end{itemize}
\end{lemma}
\begin{proof}
  Here, we consider equations of the form $\alpha s-n=\beta ns$ for
  some $\alpha,\beta\in\Z$. They are solved by induction on the number
  of prime factor of $\alpha$.

  Indeed, we first note that $s\mid n$ and $n\mid \alpha s$.
  \begin{itemize}
  \item If actually $n\mid s$, then $s=\pm n$ and $n$ satisfies either
    $(\alpha-1)n=\beta n^2$ or $(\alpha+1)n=\beta n^2$. It follows
    that $(n,s)=(0,0)$, or
    $\left(\frac{\alpha-1}{\beta},\frac{\alpha-1}{\beta}\right)$ if
    $\frac{\alpha-1}{\beta}\in\Z$, or
    $\left(\frac{\alpha+1}{\beta},-\frac{\alpha+1}{\beta}\right)$ if $\frac{\alpha+1}{\beta}\in\Z$.
  \item If $n\nmid s$ then $n=kn'$ with some prime divisor of
    $\alpha$, but then $\frac{\alpha}{k}n'-s=\beta n's$ and by
    induction, we know all such $(n_0',s_0)$ and each of them leads to a
    solution $(kn'_0,s_0)$.
  \end{itemize}
\end{proof}

\begin{lemma}\label{lem:QuadEq}
  If $m,n$ are integers such that $\big(1-m(n+4)\big)n=m\pm1$ then \[(n,m)\in\big\{(-5,-1),(-4,-3) ,(-4,-5) ,(-3,1),(-3,2),(-2,1),(-1,0) ,(-1,1),(0,-1),(0,1),(1,0)\big\}.\]
\end{lemma}
\begin{proof}
  Then $m\big(1+n(n+4)\big)=n\pm1$. So either $m=0$ or
  $\big(1+n(n+4)\big)\mid n\pm1$.
  \begin{description}
  \item[Case $m=0$] then $n=\pm1$.
  \item[Case $m\big(1+n(n+4)\big)\mid n+1$] then $1+n(n+4)\leq|n+1|$.
    \begin{description}
    \item[If $n+1\geq0$] then $n(n+3)\leq0$ so
      $n\in\{-3,-2,-1,0\}$. Only $-1$ and 0 satisfy $n+1\geq0$,
      leading to solutions $(n,m)\in\big\{(-1,0),(0,1)\big\}$.
\item[If $n+1\leq0$] then $n^2+5n+2\leq0$ so
  $n\in\left[\frac{-5-\sqrt{17}}{2},\frac{-5+\sqrt{17}}{2}\right]\cap\Z=\{-4,-3,-2,-1\}$. If
$n=-2$, then $m=\frac{1}{3}\notin\Z$. Other cases lead to solutions $(n,m)\in\big\{(-4,-3),(-3,1),(-1,0)\big\}$.
    \end{description}
\item[Case $m\big(1+n(n+4)\big)\mid n-1$] $1+n(n+4)\leq|n-1|$.
    \begin{description}
    \item[If $n-1\geq0$] then $(n+1)(n+2)\leq0$ so
      $n\in\{-2,-1\}$ and doesn't satisfy $n-1\geq0$.
\item[If $n-1\leq0$] then $n(n+5)\leq0$ so $n\in\{-5,-4,-3,-2,-1,0\}$
  leading to solutions\\ $(n,m)\in\big\{(-5,-1),(-4,-5),(-3,2),(-2,1),(-1,1),(0,-1)\big\}$.  
    \end{description}
  \end{description}
\end{proof}

\end{document}